\documentclass[
final
]{dmtcs-episciences}

\usepackage[utf8]{inputenc}
\usepackage{subfigure}

\usepackage{latexsym} 
\usepackage{amsmath} 
\usepackage{epsfig}
\usepackage{amssymb}
\usepackage{enumerate}
\usepackage{color}
\usepackage[dvipsnames]{xcolor}

\newcommand{\cB}{{\mathcal B}}

\newcommand{\cF}{{\mathcal F}}
\newcommand{\cG}{{\mathcal G}}

\newcommand{\cN}{{\mathcal N}}
 
\newcommand{\cP}{{\mathcal P}}
 
\newcommand{\cR}{{\mathcal R}} 

\newcommand{\cT}{{\mathcal T}} 

\definecolor{fashionfuchsia}{rgb}{0.96, 0.0, 0.63}

\newcommand{\blue}{\textcolor{black}}
\newcommand{\orange}{\textcolor{black}}

\newcommand{\teal}{\textcolor{black}}

\newtheorem{theorem}{Theorem}[section]
\newtheorem{lemma}[theorem]{Lemma}

\newtheorem{proposition}[theorem]{Proposition}

\newtheorem{observation}[theorem]{Observation}

\author[Katharina T. Huber et al.]{Katharina T. Huber\affiliationmark{1}  \and Simone Linz\affiliationmark{2}\thanks{Supported by the New Zealand Marsden Fund.}
  \and Vincent Moulton\affiliationmark{1}}
\title[Cherry picking in forests]{Cherry picking in forests:\\ A new~characterization for the 
	unrooted hybrid number of two phylogenetic trees}
\affiliation{
University of East Anglia, Norwich, UK\\
University of Auckland, Auckland, New Zealand}
\begin{document}
\publicationdata
{vol. 27:2}{2025}{15}{10.46298/dmtcs.11633}{2023-07-25; 2023-07-25; 2024-02-27; 2024-12-03}{2025-05-08}
\maketitle
\begin{abstract}
Phylogenetic networks are a special type of graph which generalize phylogenetic trees and 
that are used to model non-treelike evolutionary processes such as recombination
and hybridization. In this paper, we consider  
{\em unrooted} phylogenetic networks, i.e. simple, connected graphs $\cN=(V,E)$ 
with leaf set $X$, for $X$ some set of species, in which every internal vertex in $\cN$ has degree three. 
One approach used to construct such phylogenetic networks is to take as input a 
collection $\cP$ of phylogenetic trees and to look for a network $\cN$ that contains each tree in $\cP$
and that minimizes the quantity $r(\cN) = |E|-(|V|-1)$ over all such networks.
Such a network always exists, and the quantity $r(\cN)$ for an optimal network $\cN$
is called the {\em hybrid number of $\cP$}.  In this paper, we give a new characterization  
for the hybrid number in case $\cP$ consists of two trees. 
This characterization is given in terms of a {\em cherry picking sequence}
for the two trees, although to prove that our characterization holds we 
need to define the sequence more generally for two forests. 
Cherry picking sequences have been intensively studied
for collections of {\em rooted} phylogenetic trees, but our new sequences
are the first variant of this concept that can be applied in the unrooted setting. 
Since the hybrid number of two trees is equal to the well-known {\em tree bisection and \blue{reconnection} distance} 
between the two trees, our new characterization also provides 
an alternative way to understand this important tree distance.
\end{abstract}

\section{Introduction}
Phylogenetic networks are a special type of graph which generalize phylogenetic trees and 
that are used to model non-treelike evolutionary processes such as recombination
and hybridization \cite{H10}. There are several classes of 
phylogenetic networks (see e.g. \cite{KP22} for a recent review), but in this paper, we shall restrict our attention to 
{\em unrooted} phylogenetic networks (see e.g.~\blue{\cite{IK18}}). Such a network is a simple, connected graph $\cN=(V,E)$ 
with leaf set $X$, where $X$ is some set of species, \blue{and} every internal vertex in $\cN$ has degree three. 
The {\em reticulation number} $r(\cN)$ of $\cN$ is defined as $|E|-(|V|-1)$.
If $r(\cN)=0$, then $\cN$ has no cycles, in which case it is 
called an {\em (unrooted binary) phylogenetic tree} on $X$. An example
of a phylogenetic network $\cN$ with leaf set $X=\{1,\dots,6\}$
and $r(\cN)=1$ is given in the left of Fig.~\ref{fig:prelims} (the 
tree in the centre of this figure is a phylogenetic tree); formal definitions are given in the next section.

One approach commonly used to construct  an unrooted phylogenetic network is to take as input a 
set $\cP$ of phylogenetic trees and to infer
a network that embeds or `displays'~\cite{IK18, kanj08} each tree in $\cP$ (see Fig.~\ref{fig:prelims} for an example of a  phylogenetic tree displayed in a network).
Such a network always exists since, for example one can be constructed using the so-called display graph~\cite{bryant06} 
for all trees in $\cP$ and subsequently \blue{refining} the vertices of that graph whose degree is greater than three \blue{by introducing new edges}. From 
a \blue{biological} viewpoint, it is of interest to not reconstruct \blue{an arbitrary} unrooted
phylogenetic network that displays each element in $\cP$, but one whose reticulation 
number is minimized over all such networks. This 
is known as the {\em Unrooted Hybridization Number} problem ~\cite{IK18}.  If $\cN$ is an unrooted 
phylogenetic network that displays each tree in $\cP$ and whose reticulation number 
is minimized over all such networks, then $r(\cN)$ 
is denoted by $h(\cP)$ and is called the {\em hybrid number of $\cP$}.

\begin{figure}[t]
	\center
	\scalebox{0.8}{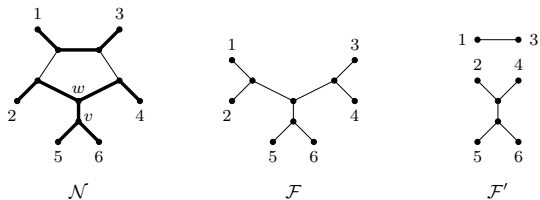}
	\caption{Left: A phylogenetic network $\cN$ with leaf set $X=\{1,\dots,6\}$. Middle: A
		forest $\cF$ that contains a single phylogenetic tree.
		Right: A forest $\cF'$ comprising two components.  Both $\cF$ and $\cF'$ are displayed by $\cN$; 
			the way in which $\cF'$ is displayed by $\cN$ is indicated in bold.
		}
	\label{fig:prelims}
\end{figure}

Intriguingly, in \cite[Theorem 3]{IK18} it is shown that, if $\cP$ contains only 
two phylogenetic trees $\cT$ and $\cT'$ on the same leaf set, then 
 $h(\cP)$ is in fact
equal to the  {\em tree bisection and \blue{reconnection} (TBR) distance}
between $\cT$ and $\cT'$ (see e.g. \cite{allen01} and also \cite{SS03}). Informally, this
distance is defined as the minimum number of so-called 
tree bisection and \blue{reconnection} moves required to convert $\cT$ into
$\cT'$, where such a move on a phylogenetic tree is essentially 
the process of cutting an edge in that tree and then 
reconnecting the resulting two trees by introducing 
a new edge between two edges, one from each of the two trees. 
The TBR distance
has been studied for several years, and is still a topic of
current research~\cite{CFS15,kelk19a,LK19,stjohn17,whidden13}. Furthermore, computing 
this distance is an NP-hard optimization problem~\cite{allen01}.
Hence, gaining a better understanding
of the hybrid number of two  phylogenetic trees is not only of
interest for constructing phylogenetic networks but
also for shedding new light on computing the TBR distance between them.

In this paper, we establish a new characterization for the hybrid number of two  
phylogenetic trees  (Theorem~\ref{main}) and, therefore the TBR distance (Theorem~\ref{theo:tbr}).
Our approach is motivated by the concept of a {\em cherry picking
sequence} that was introduced in~\cite{HLS13} in the context of two {\em rooted} (binary) phylogenetic trees.
A {\em cherry} in such a tree is a pair of
	leaves that are adjacent to the same vertex.
Roughly speaking, for a pair of rooted  phylogenetic
trees $\cR$ and $\cR'$ on the same leaf set $X$, a cherry 
picking sequence for $\cR$ and $\cR'$ is an ordering or sequence \blue{$\sigma$} of the elements in 
$X$, say $\blue{\sigma}=(x_1,x_2,\ldots,x_n)$, $n=|X|$, such that 
each $x_i$, $i\in\{1,2,\ldots,n\}$, is a leaf 
in a cherry in the phylogenetic trees obtained from $\cR$ and $\cR'$ by pruning off $x_1,x_2,\ldots,x_{i-1}$.
In \cite{HLS13} it is shown that a cherry picking sequence 
for two rooted phylogenetic trees $\cR$ and $\cR'$ exists 
precisely if  $\cR$ and $\cR'$ can be embedded in a {\em rooted} binary phylogenetic network $\cN$ on $X$ 
that satisfies certain structural properties (namely it is `time consistent' and `tree-child').
The number of elements $x_i$ in $\sigma$ with $i\in\{1,2,\ldots,n\}$ 
for which the two trees obtained from $\cR$ and $\cR'$ by pruning off $x_1,x_2,\ldots,x_{i-1}$  
have cherries $\{x_i,y_i\}$ and $\{x_i,y_i'\}$ such that $y_i\ne y_i'$
is called the {\em weight} of $\sigma$. In addition, it is shown in the same paper
that the minimum weight over all cherry picking sequences for $\cR$ and $\cR'$ 
equates to the minimum number of reticulations (vertices with in-degree two and out-degree one) 
that are required to display $\cR$ and $\cR'$ 
by a rooted phylogenetic network that is time consistent and tree-child. 
\blue{Since the publication of~\cite{HLS13}, cherry picking sequences
have been extensively studied and generalized to larger classes of rooted phylogenetic networks and  
arbitrarily large collections of rooted phylogenetic trees 
(see e.g.~\cite{janssen21,LS19}). This theoretical work has, in turn, resulted in the development of practical algorithms to \teal{reconstruct} rooted phylogenetic networks from a set of rooted phylogenetic trees~\cite{bernardi,vanIersel22}. In a related line of research, cherry picking operations have recently been used to define and compute distances between phylogenetic networks~\cite{landry23, landry}.}

In this paper, we shall introduce an \blue{analog of a
cherry picking sequence} for two unrooted phylogenetic trees $\cT$ and $\cT'$ on the same leaf set. In particular, 
the main result of this paper (Theorem~\ref{main}) proves that the hybrid number and, consequently, 
the TBR distance of $\cT$ and $\cT'$ can be given in 
terms of the minimum weight of a cherry picking sequence using 
an alternative weight to the rooted setting, and
the minimum is again taken over all such sequences for $\cT$ and $\cT'$. Interestingly, to obtain this
result we found that it is necessary to define a cherry picking sequence for
two {\em forests} on the same leaf set (a forest is a collection of phylogenetic trees, see e.g. \blue{the right-hand side of} Fig.~\ref{fig:prelims}), 
and then apply such a sequence
to the special case where the two forests are both phylogenetic trees.
Intuitively, this is the case because, instead of only pruning 
leaves of cherries as in the rooted case  that is described in the previous paragraph, a 
cherry picking sequence whose weight characterizes the 
hybrid number of $\cT$ and $\cT'$ allows for operations that split a 
phylogenetic tree into two smaller phylogenetic trees by deleting edges that are not necessarily pendant.
It is also interesting to note that our definition of a  cherry picking sequence
has some similarity with certain data reduction rules that 
have been recently introduced for establishing 
parameterized algorithms to compute the TBR distance 
\blue{between two phylogenetic trees} \cite{LK19}.
	
We now summarize the \blue{content} of the rest of this paper.
In Section~\ref{sect:prelim}, we present some preliminaries.
Subsequently, in Section~\ref{sect:picking}, we introduce the main concept
of a cherry picking sequence for two forests, and show that such a sequence 
always exists for any pair of forests. In Section~\ref{sect:bound1},
we show that if $\sigma$  is  a cherry picking sequence for two forests $\cF$ 
and $\cF'$ on $X$, then there is an unrooted phylogenetic network $\cN$ on $X$
that displays $\cF$ and $\cF'$ such that $r(\cN)$ is at most the weight of $\sigma$ (Theorem~\ref{uppercherry}).
In Section \ref{sect:tech}, we prove \blue{four}  technical lemmas  that are concerned with how forests are displayed in networks.
These lemmas allow us to complete the proof of our main result (Theorem~\ref{main})
in Section~\ref{sect:bound2}, where we prove that if $\cN$ is a phylogenetic network that 
displays two forests $\cF$ and $\cF'$ on $X$, then 
there is a cherry picking sequence for $\cF$ and $\cF'$ whose weight is at most $r(\cN)$
(Theorem~\ref{uppernet}).
We conclude in Section~\ref{sect:conclude} with a brief discussion
of some future directions of research.

\section{Preliminaries}
\label{sect:prelim}

\noindent In this section we present some basic notation and definitions.\\

\noindent{\bf Graphs.}
   \teal{The graphs that we shall consider in this paper will be multi-graphs. 
   We denote such a graph $G$ by an ordered pair $G=(V,E)$, where $V=V(G)$ are the vertices of $G$ and
   $E=E(G)$ are the edges of $G$. No graph considered will contain a loop, but 
   it may contain a {\em multi-edge}, that is, a pair of
   vertices which are joined by more than one edge.
	In case there is a single edge in $G$ joining
	two vertices $u$ and $v$ we call it an {\em edge} and denote it by $\{u,v\}$.
    Note that if $E(G)$ contains only edges, then $E(G)$ is a set.} 
	An edge $e$ of a graph $G$ is called a {\em cut-edge} of $G$ if the deletion
	of $e$ disconnects $G$. If $e$ is a cut-edge that contains a 
	vertex of $G$ of degree one then we call $e$ a {\em trivial cut-edge}.  
	A {\em subdivision} of $G$ is a graph obtained from $G$ by replacing edges in $G$ by 
	paths containing at least one edge. If $G$ has two edges $\{u,v\}$, $\{v,w\}$, $u,v,w \in V(G)$ and $v$ has degree two \blue{and is not incident with two edges $\{u,v\}$ or two edges $\{v,w\}$},  
	the process of replacing these two edges by a single edge $\{u,w\}$ and removing $v$ is called 
	{\em suppressing} the vertex $v$. If $G$ has a multi-edge between two vertices 
	$u$ and $v$ in $V(G)$, the process of replacing the edges by a single edge $\{u,v\}$ is 
	called  {\em suppressing} the multi-edge.
    A {\em leaf} of $G$ is a vertex in $V(G)$ \blue{of degree at most one.}
    \blue{The set of leaves of $G$ is denoted by $L(G)$. Finally, we define $r(G)=|E|-(|V|-1)$ and refer to $r(G)$ as the {\it reticulation number of $G$}} which is also known as the cyclomatic number of $G$ \cite{berge01}.\\

\noindent {\bf Phylogenetic trees and networks.}
\label{sec:networks-trees-forests}
\blue{Throughout the paper, $X$ denotes} a finite set with $n=|X| \ge 1$, where
we think of $X$ as a set of species.

A {\em phylogenetic network} $\cN$ on $X$ is a simple, connected graph $(V,E)$, 
with leaf set $X \subseteq V$ and such that all non-leaf vertices have degree \blue{three} \cite{IK18}. 
Note that this type of graph is sometimes called a {\em binary} phylogenetic network.
Also note that if $X= \{x\}$, then we shall regard the graph
consisting of the single vertex $x$ as being a phylogenetic network, and will denote it by $x$.	
If $\cN$ is a tree, that is, it contains no cycles, then we call $\cN$ a {\em phylogenetic tree (on $X$)}.
Suppose \blue{that} $\cT$ is a phylogenetic tree on $X$ and $Y\subseteq X$ is a non-empty subset. 
We denote by $\cT(Y)$ the minimal subtree of $\cT$ that connects the
elements in $Y$.  Note that $\cT(Y)$ need not be a phylogenetic tree on $Y$ because
it might contain vertices that have degree two. We therefore denote the phylogenetic tree 
obtained by suppressing all vertices in $\cT(Y)$ of degree two by $\cT|Y$ and refer to this
as being the {\em restriction} of $\cT$ to $Y$. 

Suppose that $\cT$ is a phylogenetic tree and that $\cN$ is a phylogenetic network. 
We call $\cT$ a {\em pendant subtree} of $\cN$ if either $\cN$ equals $\cT$ or
there is a cut-edge $e$ in $\cN$ so that $\cT$ \blue{is equal to one of the two connected components obtained from $\cN$ by deleting $e$ and suppressing any resulting degree-two vertices.}
In the latter case, we shall also refer to $\cT$ as being a {\em proper} pendant subtree of $\cN$,
and we denote the cut-edge $e$ that gives rise to $\cT$ 
by $e_{\mathcal T,\mathcal N}$ or just $e_{\mathcal T}$ if $\cN$ is
clear from the context. If $\cT'$ is a phylogenetic tree, and  $\cT$ is a 
proper pendant subtree of $\cT'$, we let  $\cT'-\cT$ denote
the  phylogenetic tree on \blue{$X\setminus L(\cT)$} obtained by deleting $\cT$ and $e_{\cT}$ from $\cT'$ 
and suppressing the resulting degree-\blue{two} vertices (in case there are any).

An edge of $\cN$ that contains a leaf $x$ of $\cN$ is called a {\em pendant edge} in $\cN$ (so 
that, in particular, $x$ is a proper pendant subtree of $\cN$ and $e_x$ is the pendant edge of $x$).
In case $\cN$ has at least two leaves, a {\em cherry} of $\cN$ consists of two \blue{distinct} leaves $x$ and $y$ of $\cN$
such that either $\cN$ \blue{is equal to the edge $\{x,y\}$,} or $x$ and $y$ are adjacent to a common vertex
(so that, in particular, in the latter case the phylogenetic tree $\{x,y\}$ is a proper pendant subtree of $\cN$).
We shall denote a cherry consisting of leaves $x$ and $y$ by $(x,y)$, where 
the order of $x$ \blue{and} $y$ is unimportant.  We call $\cN$ {\em pendantless} if it contains no proper 
pendant subtree with at least two leaves (e.g. for the phylogenetic network $\cN$ 
in Fig.~\ref{fig:prelims}, since the subtree with leaf set $\{5,6\}$ is a cherry of $\cN$, it 
follows that $\cN$ is not pendantless).
We also need notation for two further special cases.
Suppose \blue{that} $\cT$ is a pendant subtree of $\cN$. If $\cT$ has leaf set $\{x,y,z\} \subseteq X$ and $\cT$ is a 
proper pendant subtree of $\cN$ such that $(x,y)$ is a cherry of $\cN$, then we denote $\cT$ 
 by $((x,y),z)$ (for brevity, if $\cT=\cN$, then we shall also denote \blue{$\cT$} by  $((x,y),z)$ although the choice of the 
 cherry $(x,y)$ is unimportant).
In addition, if $\cT$ has leaf set $\{x,y,z,w\}\subseteq X$ 
and $(x,y)$, $(z,w)$ are both cherries in $\cN$, then we denote $\cT$ by $((x,y),(z,w))$.\\

\newpage

\noindent {\bf Forests.}
A {\em forest} $\cF$ on $X$ is a set of phylogenetic trees whose collective leaf set equals $X$ \blue{and no two leaves have the same label}.
Abusing notation, we will consider a phylogenetic tree \blue{$\cT$} as also being a forest (i.e. the
singleton set that consists of \blue{$\cT$}).
A pair $(x,y)$ with $x,y\in X$ distinct is called a {\em cherry} of a forest $\cF$ if it is a cherry in
\blue{one} of the trees in $\cF$. 

For an edge $e$ in a phylogenetic tree $\cT$ on $X$
we denote by $\cT-e$ the forest obtained from $\cT$ by deleting $e$
and suppressing degree-\blue{two} vertices (if any) in the resulting trees.
If $\cF$ is a forest on $X$, $|X|\ge 2$, and $x\in X$, then we 
let $\cT_x$ denote the tree in $\cF$ that contains $x$ in its leaf set.
We denote by $\cF -x$ the forest obtained from $\cF$ by 
removing $\cT_x$ from $\cF$ in case $|L(\cT_x)|=1$ and replacing $\cT_x$
by $\cT_x-x$ otherwise. Also, for $e$ an edge
in $\cF$, we let $\cT_e$ denote the tree in $\cF$ with $e$ in its edge set, and  
denote by $\cF-e$ the forest obtained by replacing $\cT_e$ by $\cT_e-e$.
For example, referring to Fig.~\ref{fig:prelims} for $x=4$, the 
tree $\cT_x$ in $\cF'$ is $((2,4),(5,6))$ and $\cF'-4$ is the forest 
comprising the cherry \blue{with leaf set} $\{1,3\}$ and \blue{the three-leaf tree} $((5,6),2)$. \blue{Similarly, for $e$ being the pendant edge incident with $4$,} the tree $\cT_e$
in $\cF'$ is again $((2,4),(5,6))$ and $\cF-e$ is the 
forest comprising the cherry \blue{with leaf set} $\{1,3\}$, the isolated vertex 4, and \blue{the three-leaf tree} $((5,6),2)$.\\

\noindent {\bf Displaying phylogenetic trees and forests.}
Suppose \blue{that} $\cN$ is a phylogenetic network on $X$, and that $\cT$ 
is a phylogenetic tree on some subset $Y \subseteq X$. Then 
we say that $\cT$ is {\em displayed} by $\cN$ if $\cT$ can be obtained from  
a subtree $\cN[\cT]$ of $\cN$ by suppressing all vertices in $\cN[\cT]$ 
with degree two. We shall refer to $\cN[\cT]$ as 
 \blue{an} {\em image} of $\cT$ in $\cN$. 
 For example in Fig. ~\ref{fig:prelims}, for $\cT=((2,4),(5,6))$, \blue{a} 
 	subgraph $\cN[\cT]$ of $\cN$ with leaf set $\{2,4,5,6\}$ is indicated in bold.
Note that an image of $\cT$ in $\cN$ is isomorphic to a subdivision of $\cT$
and that $\cT$ could have several images in $\cN$.
Usually the image of $\cT$ in $\cN$ that we are considering is clear from the context, but in case we want 
to make this clearer we shall refer to it explicitly.
 
Suppose that $\cF$ is a forest on $X$. Then we say that $\cN$ {\em displays} $\cF$ if
every tree in $\cF$ is displayed by $\cN$, and  we can choose an image 
$\cN[\cT]$ for each tree $\cT$ in $\cF$ such that for any distinct trees $\cT, \cT' \in \cF$ the 
images $\cN[\cT]$ and $\cN[\cT']$ do not share a vertex. \blue{If $\cN$ displays $\cF$, then} we refer to
the set $\cN[\cF]=\{ \cN[\cT]\,:\, \cT\in\cF\}$ as an {\em image} of $\cF$ in $\cN$. 
	For example, considering again the \blue{forest $\cF'$} 
		and the phylogenetic network $\cN$ pictured in Fig.~\ref{fig:prelims},
	the subgraph $\cN[\cF']$ of $\cN$ is indicated in bold.
	
\section{Cherry picking sequences}
\label{sect:picking}

\begin{figure}[h!]
	\center
	\scalebox{0.97}{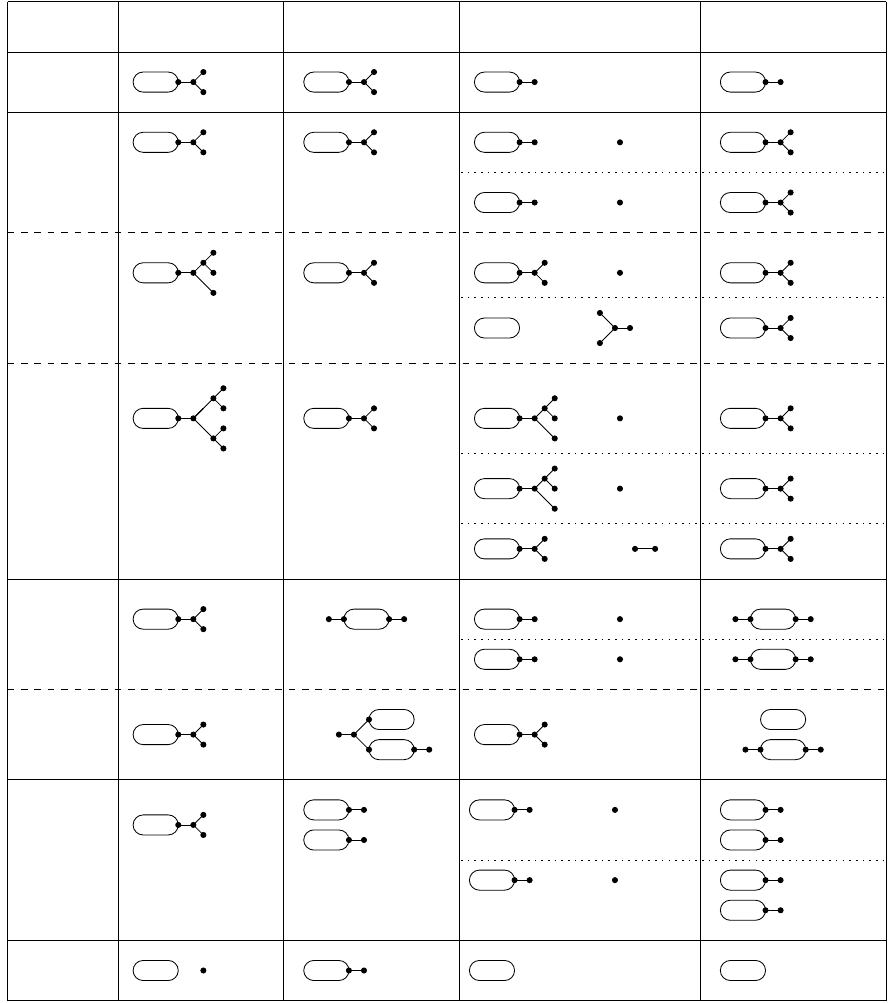}
	\caption{\blue{The different cases as described in the definition of a cherry picking sequence. Ovals indicate subtrees. The roles of $\cF$ and $\cF'$ in (C2) and (C3) could also be reversed.}}
	\label{fig:cps-def}
\end{figure}

In this section, we define the concept of a cherry picking sequence for two forests
	having the same leaf set, and show that such a sequence
	always exists for any pair of forests. We begin by making a simple but important observation, whose proof is straight-forward and omitted.

\begin{observation}\label{note} 
If we are given two forests $\cF$ and $\cF'$ on $X$
and two \blue{distinct} elements $x,y\in X$ such that  $(x,y)$ is a cherry in $\cF$ then precisely one of  the following must hold: 
(i) $(x,y)$ is a cherry in $\cF'$; 
(ii) $(x,z)$ or $(y,z)$ is a cherry in $\cF'$ \blue{for} some $z \in \blue{X\setminus\{x,y\}}$,
(iii) $x$ and $y$ are in the same tree in $\cF'$, but 
neither $x$ nor $y$ is in a cherry of $\cF'$; or
(iv) $x$ and $y$ are in different trees in $\cF'$,
but neither $x$ nor $y$ is in a cherry of $\cF'$.
\end{observation}
%

Exploiting this observation, we shall now define a 
	cherry picking sequence for  two forests $\cF$ and $\cF'$ on the same leaf set $X$ as follows: 
	We call a sequence $\sigma = (x_1,x_2,\dots,x_m)$, \blue{$m\geq |X|\geq 1$}, of elements in $X$ 
	a {\em cherry picking sequence for $\cF$ and $\cF'$ (of length $m$)} if 
	there are forests $\cF[i]$ and $\cF'[i]$, $1\le i\le m$, such that $\cF[1]=\cF$, $\cF'[1]=\cF'$, 
	$\cF[m] = \cF'[m] =\{x_m\}$ and, for each $1 \le i \le m-1$, precisely one of the  
	following cases holds. \blue{The different cases are illustrated in Fig.~\ref{fig:cps-def} and an example of a cherry picking sequence is presented in Fig.~\ref{fig:cherry-pick-ex}}.

\begin{enumerate}[(C1)]
\item $(x_i,y)$ is a cherry in both $\cF[i]$ and $\cF'[i]$, 
and $\cF[i+1] =\cF[i]-x_i$ and $\cF'[i+1] =\cF'[i] -x_i$;
\item$(x_i,y)$ is a cherry in $\cF[i]$ and one of the following three cases holds:
\begin{enumerate}[(a)]
\item $(x_i,z)$ is a 
cherry in $\cF'[i]$ with $y \neq z$, $\cF'[i]=\cF'[i+1]$ and either 
\begin{enumerate}[(i)]
\item $\cF[i+1]$ equals $\cF[i]-e_{x_i}$ or $\cF[i]-e_{y}$; or
\item $((x_i,y),p)$ is a proper pendant subtree of a tree in $\mathcal F[i]$, $p \in X$, 
and $\cF[i+1]=\cF[i]-e_{\mathcal S}$ for some $\mathcal S \in \{p,((x_i,y),p)\}$; or
\item $((x_i,y),(p,q))$ is a proper pendant subtree of a tree in $\mathcal F[i]$, $p \neq q \in X$, and  
$\cF[i+1]=\cF[i]-e_{\mathcal S}$ for some \blue{$\mathcal S \in \{p,q,(p,q)\}$};
\end{enumerate}
\item $x_i$ and $y$ are both leaves in some tree $\cT' \in \cF'[i]$ 
and neither $x_i$ nor $y$ is contained in a cherry in $\cT'$, and either
\begin{enumerate}[(i)]
\item $\cF'[i+1]=\cF'[i]$ and $\cF[i+1]=\cF[i]-e_{x_i}$ or $\cF[i+1]=\cF[i]-e_{y}$; or 
\item for $e$ the edge in $\cT'$ that contains the vertex adjacent to $x_i$ and 
is not on the path 	between $x_i$ and $y$,
$\cF'[i+1]= \cF'[i] -e$ and $\cF[i+1]=\cF[i]$;
\end{enumerate}
\item $x_i$ and $y$ are contained in different trees of $\cF'[i]$, neither
of them is an isolated vertex or contained in a cherry of $\cF'[i]$,  $\cF'[i]=\cF'[i+1]$, and
		$\cF[i+1]=\cF[i]-e_{x_i}$ or $\cF[i+1]=\cF[i]-e_y$;
\end{enumerate}
or the same holds with the roles of $\cF$ and $\cF'$ reversed.

\item $x_i$ is a component of $\cF[i]$, 
		$\cF[i+1] = \cF[i] - x_i$, and $\cF'[i+1]=\blue{\cF'[i]} - x_i$,
		or the same holds with the roles of $\cF$ and $\cF'$ reversed.\\
\end{enumerate}	

 	\begin{figure}[t!]
		\center
		\scalebox{0.73}{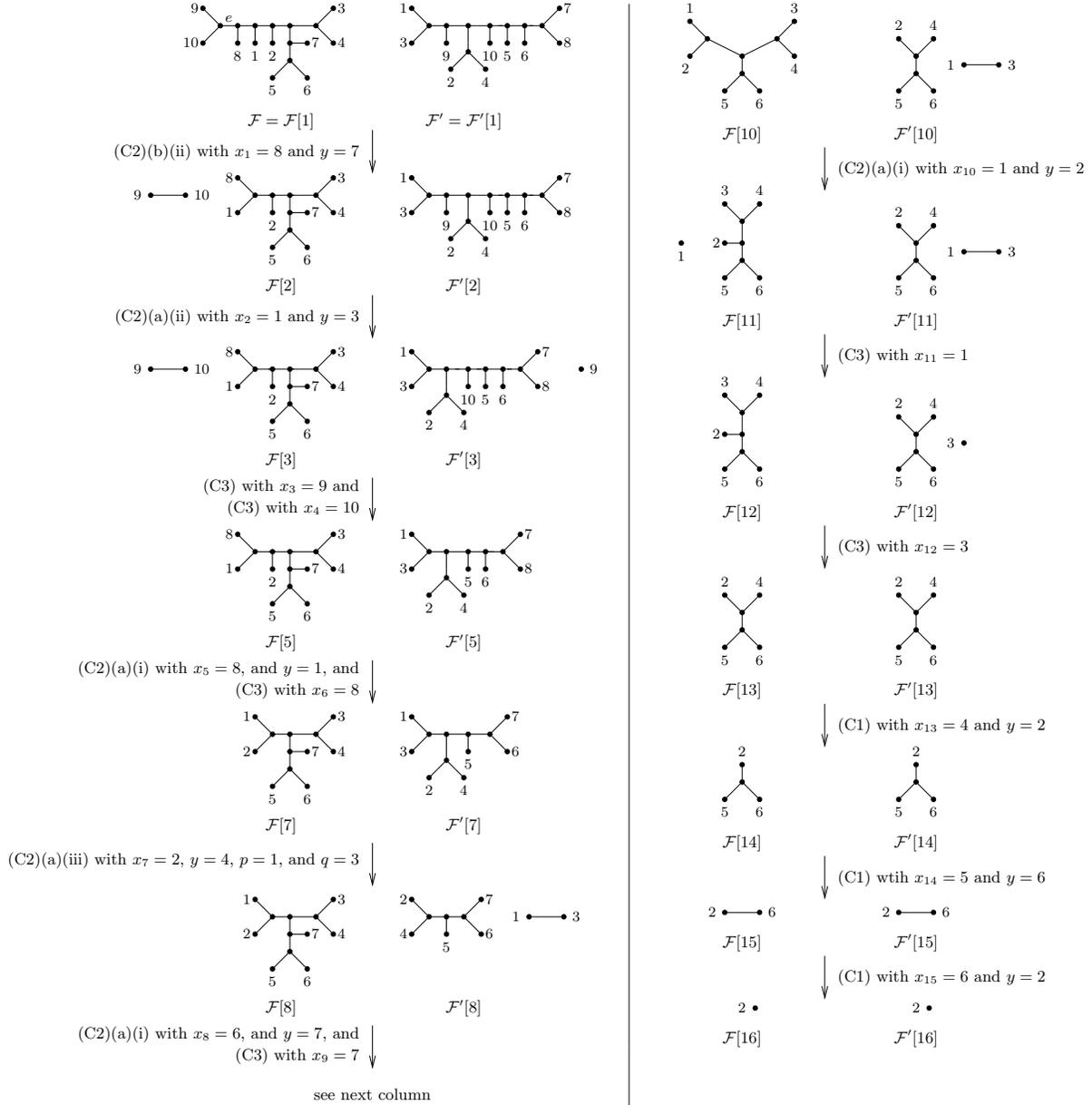}
		\caption{For the two depicted forests
			$\cF$ and $\cF'$ the sequence \teal{$(8,1,9,10,8,8,2,6,7, 1,1,3,4,5,6,2)$} 
            is a cherry picking sequence for 
			$\cF$ and $\cF'$. The reductions applied within the sequence are indicated \teal{next to} the arrows. }
		\label{fig:cherry-pick-ex}
	\end{figure}

\noindent Note that in cases (C1) and (C3), $\cF[i+1]$ and $\cF'[i+1]$ both have one leaf less than $\cF[i]$ and $\cF'[i]$, \blue{respectively,} whereas in all sub-cases of (C2)  
		one of $\cF[i+1]$ and $\cF'[i+1]$ is obtained by removing an edge from $\cF[i]$ or $\cF'[i]$.
		For simplicity, we shall call the removal of a leaf or an edge as described in each of the cases in (C1)--(C3) a {\em reduction},
		or, if we want to be more specific, we may, for example, say that 
		we use a (C2)(a)(ii)-reduction (applied to $x_i$).

\blue{Let $\sigma=(x_1,\dots,x_m)$ be a
cherry picking sequence for two forests $\cF$ and $\cF'$ with \blue{$m\geq 2$}. For each $1\le i \le m-1$, we set $c(x_i)=0$
if one of the reductions in (C1) and (C3) is applied \blue{to $x_i$} and, otherwise, $c(x_i)=1$}. In addition, we  define 
the  {\em weight} of $\sigma$  by $w(\sigma)=\sum_{i=1}^{m-1} c(x_i)$ if $m\geq 2$ and $0$ if $m=1$. 	
To illustrate these definitions, let  $X=\{1,\ldots, 6\}$
and consider the two forests $\cF$ and  $\cF'$ on  $X$ in Fig.~\ref{fig:prelims}.
Then \teal{$(8,1,9,10,8,8,2,6,7, 1,1,3,4,5,6,2)$} 
gives a cherry picking  sequence for $\cF$ and $\cF'$ with weight \teal{6}, as can be verified 
by looking at the reductions depicted in Fig.~\ref{fig:cherry-pick-ex}.
	
We now show that a cherry picking sequence always exists for any pair of forests with the same leaf set.

\newpage

\begin{proposition}\label{CPexists}
	Suppose that $\cF$ and $\cF'$ are forests on $X$.
	Then there exists a cherry picking sequence $\sigma$ 
	for $\cF$ and $\cF'$ of length $m$ \blue{for} some \blue{$m \ge  |X|$}.	
\end{proposition}
\begin{proof}
Suppose that $\cF$ and $\cF'$ are two forests on $X$, $n=|X| \ge 1$. 
We use induction on $n$. If $n=1$, then we can assume $X=\{x\}$, and 
so $\sigma=(x)$ is a cherry picking sequence
for $\cF$ and $\cF'$ (since $\cF = \cF' = \{x\}$).

So suppose that $n >1$, and that for any two forests $\cF_1$ 
and $\cF_2$ on a set $Y$ with $1 \le |Y| \le n-1$,
there exists a cherry picking sequence for $\cF_1$ and $\cF_2$ 
of length at least $ |Y|$.

First, suppose that there is no cherry in either $\cF$ or $\cF'$. Then the 
edge sets of $\cF$ and $\cF'$ must both be empty. Pick any $x \in X$
and apply a (C3)-reduction
to obtain forests $\cF_1=\cF-x$ and $\cF_2=\cF'-x$ on \blue{$X \setminus \{x\}$}.
By induction, there is a cherry picking sequence $(x_1,\dots,x_p)$ 
for  $\cF_1$ and $\cF_2$ with $p \ge n-1$. Hence $(x,x_1,\dots,x_p)$
is a cherry picking sequence for $\cF$ and $\cF'$ with length at least $n$. 

Now suppose that there exist $x,y\in X$ such that $(x,y)$
is a cherry in one of the forests, say $\cF$. If $x$ or $y$ is a component 
of $\cF'$ then, as in the previous case, applying a (C3)-reduction  
and our induction hypothesis yields
a cherry picking sequence for $\cF$ and $\cF'$. So assume that neither $x$ nor $y$ is a component of $\cF'$.
Then, by Observation~\ref{note}, setting $\cF[i]= \cF$ and $\cF'[i]= \cF'$, 
without loss of generality we must be in the situation given in (C1)
or one of the cases given in (C2) in the definition of a cherry picking sequence with $x=x_i$.	

In case (C1), we can apply the reduction $\cF_1=\cF[i]-x$ and $\cF_2=\cF'[i]-x$ to $\cF[i]$ and $\cF'[i]$.
By induction, there is a cherry picking sequence $(x_1,\dots,x_p)$ 
for  $\cF_1$ and  $\cF_2$ with $p \ge n-1$. Hence $(x,x_1,\dots,x_p)$
is a cherry picking sequence for $\cF$ and $\cF'$ with length $p+1$ which is at least $n$.

For case (C2), since one of the cases (a), (b) or (c) must hold,
we can apply the reduction $\cF_1=\cF[i]-e_x$ and $\cF_2=\cF'[i]$
(which occurs in cases (a)(i), (b)(i) and (c)). 
This implies that $x$ is a component in $\cF_1$.
So, we can then apply a (C3)-reduction to $\cF_1$ and $\cF_2$ with $x_i=x$ which 
results in two forests $\cF_1''$ and $\cF_2''$ on \blue{$X\setminus\{x\}$}. 
By induction, there is a cherry picking sequence $(x_1,\dots,x_p)$ 
for  $\cF_1''$ and  $\cF_2''$ with $p \ge n-1$. It now follows that $(x,x,x_1,\dots,x_p)$
is a cherry picking sequence for $\cF$ and $\cF'$ with length $p+2 > n$.
\end{proof}

Note that in the proof of Theorem~\ref{CPexists}, we did not need to 
apply either a (C2)(a)(ii), a (C2)(a)(iii), or a (C2)(b)(ii)-reduction. However, as
we shall see below, we require these additional reductions to ensure that
we obtain a cherry picking sequence with the desired weight in the proof of Theorem~\ref{uppernet}.

\section{Bounding the hybrid number from above using cherry picking sequences}
\label{sect:bound1}

The {\em hybrid number} $h(\cF,\cF')$ of two forests $\cF$ and $\cF'$ on $X$ is defined to be 
$$
h(\cF,\cF') = \min\{ r(\cN) \,:\, \cN \mbox{ displays } \cF \mbox{ and } \cF'\},
$$
where $r(\cN)$ denotes the reticulation number of $\cN$ \blue{defined in Section~\ref{sect:prelim}}.
\blue{By Proposition~\ref{CPexists} and Theorem~\ref{uppercherry} (below), note 
that there always exists a phylogenetic network \blue{on $X$} that displays $\cF$ and $\cF'$.}
In this and the next two sections, we shall focus on proving the main result of this paper:

\begin{theorem}\label{main}
Suppose that $\cF$ and $\cF'$ are forests on $X$. Then 
$$
h(\cF,\cF') = \min\{ w(\sigma) \,:\, \sigma \mbox{ is a cherry picking sequence for } \cF \mbox{ and } \cF'\}.
$$
\end{theorem}

\noindent To prove Theorem~\ref{main}, first note that the right-hand side of
the equality in Theorem~\ref{main} exists by  Proposition~\ref{CPexists}. In 
this section, we shall prove a result (Theorem~\ref{uppercherry} below)
	from which it directly follows that the hybrid number for two forests can be no larger 
	than the weight of an optimal cherry picking sequence for \blue{$\cF$ and $\cF'$}.
	Then, after proving \blue{some} supporting lemmas in the next section, in Section~\ref{sect:bound2}
	we shall prove a result
	(Theorem~\ref{uppernet}), from which it directly follows that the weight of an optimal cherry picking sequence for
	two forests is no larger than the hybrid number for the forests. The proof of Theorem~\ref{main} then follows immediately.

Before proceeding, denoting by $d_{TBR}(\cT,\cT')$  the TBR distance  between
two phylogenetic trees $\cT$ and $\cT'$ on $X$ \cite{allen01,SS03},  we note that as an immediate corollary of Theorem~\ref{main} and
\cite[Theorem 3]{IK18} we obtain:

\begin{theorem} \label{theo:tbr}
	Suppose that $\cT$ and $\cT'$ are phylogenetic trees in $X$. Then 
	$$
	d_{TBR}(\cT,\cT') = \min\{ w(\sigma) \,:\, \sigma \mbox{ is a cherry picking sequence for } \cT
	\mbox{ and } \cT'\}.
	$$
\end{theorem}

\blue{To prove Theorem~\ref{uppercherry} we will 
	use the following two lemmas.}

\begin{lemma}\label{reviewer1}
	\blue{Suppose that $\cF$ and $\cF'$ are forests on $X$, $|X|\ge 2$, \teal{such}
		that either (i) both $\cF$ and $\cF'$ contain a cherry $(x,y)$, $x,y \in X$, 
		or (ii) $\cF$ contains a component $x$, $x \in X$.
		If $\cN'$ is a phylogenetic network 
  \blue{on $X\backslash \{x\}$} that displays
		$\cF-x$ and $\cF'-x$, then there is a phylogenetic
		network $\cN$ \blue{on $X$} with $r(\cN)=r(\cN')$ that displays $\cF$ and $\cF'$.}
\end{lemma}
\begin{proof} 
\blue{We obtain a phylogenetic network $\cN$ from $\cN'$ as follows.
For (i) we insert a new pendant edge containing
$x$ into the pendant edge in $\cN'$ containing $y$ to obtain $\cN$.} 

For (ii), suppose that $x$ is \blue{a} leaf in some tree $\cT'$ in $\cF'$. 
If $\cT'$ is equal to $x$, then we  insert a new pendant 
edge containing $x$ into any edge of $\cN'$
to \blue{obtain $\cN$}, and 
if $\cT'$ is the cherry $(x,y)$, $y\in X$, then we
insert a new pendant edge containing
$x$ into the pendant edge in $\cN'$ containing $y$ to \blue{obtain $\cN$}. Now, suppose $|L(\cT')|\geq 3$. Then there exists a vertex $u$ in $\cT'$ 
that is adjacent with $x$ and two further such edges in $\cT'$.
Let $e$ be the edge in $\cT'-x$ that 
results from removing the pendant edge containing $x$ from $\cT'$ \blue{and}
suppressing $u$. 
\blue{We then} insert a pendant edge with leaf $x$
anywhere into an edge contained in the image of $e$ in $\cN'$
to \blue{obtain $\cN$}.

\blue{In all of these cases, clearly $\cN$ displays $\cF$ and $\cF'$, and $r(\cN)=r(\cN')$.}
\end{proof}

\begin{lemma}\label{reviewer2}
\blue{Suppose that $\cF$ and $\cF'$ are forests on $X$, $|X| \ge 2$, and that $e$ is an 
edge in $\cF'$. If $\cN'$ is a phylogenetic network \blue{on $X$} that displays
$\cF$ and $\cF'-e$, then there is a phylogenetic
network $\cN$ \blue{on $X$} that displays $\cF$ and $\cF'$ with $r(\cN)=r(\cN')+1$.}
\end{lemma}
\begin{proof} 
\blue{Suppose that $\cN'$ displays $\cF$ and $\cF'-e$. If $e$ 
is an isolated cherry $\{x,y\}$ in $\cF'$, $x,y \in X$, then insert a new edge
into $\cN'$ whose end vertices subdivide the two pendant edges in $\cN'$
that contain $x$ and $y$. If $e=\{u,x\}$ contains a leaf $x$, $x \in X$, and
a vertex $u$ with degree 3 in $\cF'$, then let $f$ be the edge in 
$\cF'-e$ that results from suppressing $u$ after removing $e$ from $\cF'$, 
and insert a new edge into $\cN'$ whose end vertices subdivide the 
pendant edge containing $x$ in $\cN'$ and any edge in the image of 
the edge $f$ in $\cN'$. And, if $e=\{u,v\}$,
with $u$ and $v$ both having degree 3 in $\cF$, 
then let $f$ and $g$ be the two edges in $\cF'-e$ that 
result from suppressing the vertices $u$ and $v$ in $\cF-e$, 
and insert a new edge into $\cN'$ whose end vertices subdivide
any edge in the image of $f$ in $\cN'$ and any edge in the 
image of $g$ in $\cN'$ (noting that the images of $f$ and 
$g$ must be disjoint since $\cN'$ displays $\cF-e$, and 
$f$ and $g$ are contained in different components of $\cF-e$). }

\blue{In all three cases it is straight-forward to see that 
$\cN$ is a phylogenetic network that displays $\cF$ and $\cF'$, and 
that $r(\cN)=r(\cN')+1$.}
\end{proof}

\begin{theorem}\label{uppercherry}
	Suppose that $\sigma$  is a cherry picking sequence for two forests $\cF$ 
	and $\cF'$ on $X$. Then there is a phylogenetic network $\cN$ on $X$
	that displays $\cF$ and $\cF'$ with $w(\sigma) \ge r(\cN)$.
\end{theorem}

\begin{proof} 
\blue{Clearly, we can assume $|X|\ge 2$.
Let $\sigma=(x_1,x_2,\dots,x_m)$, $m\geq 2$,} 
so that $\cF[1]=\cF$ and $\cF'[1]=\cF'$ in the definition of a cherry picking sequence.
We use induction on the weight $w(\sigma)$ of $\sigma$.

\noindent \blue{Base Case: 
Suppose $w(\sigma) =0$. 
Then only the reductions in (C1) or (C3) are applied at stage $i$ in $\sigma$, for all $1 \le i \le m-1$. 
Now, clearly the two forests $\cF[m-1]$ and $\cF'[m-1]$ 
are displayed by the phylogenetic tree $\{x_{m-1},x_m\}$. So, by applying 
Lemma~\ref{reviewer1} $(m-2)$-times it follows that $\cF$ and $\cF'$ are
displayed by some phylogenetic tree $\cT$ on $X$. Since $w(\sigma)=r(\cT)=0$,  
the base case holds.}

\blue{Now suppose that the theorem holds for all cherry picking 
sequences $\sigma'$ for two forests with $0 \le w(\sigma') \le w(\sigma)-1$, and that $w(\sigma) \ge 1$.
Let $1 \le i \le m-1$  be the smallest $i$ such that $c(x_i) =1$, i.e. $i$ is the first time that
we apply a (C2)-reduction in $\sigma$.}

\blue{Suppose $i > 1$. If 
we can find a phylogenetic network $\cN'$ that displays $\cF[i]$ and $\cF'[i]$
so that, for the cherry picking sequence
$\sigma' = (x_i,\dots,x_m)$ for $\cF[i]$ and $\cF'[i]$, we have
$w(\sigma')\ge r(\cN')$, then by applying Lemma~\ref{reviewer1} $(i-1)$ times
to the sequence $(x_1,\dots,x_{i-1})$, it follows that there is a 
phylogenetic network $\cN$ that displays $\cF$ and $\cF'$ with
$w(\sigma) = w(\sigma') \ge r(\cN')= r(\cN)$. Thus we may
assume without loss of generality that $i=1$.}

\blue{So, let $\sigma =(x_1,\dots,x_m)$, $m \ge 2$, be a cherry picking sequence
for $\cF$ and $\cF'$ with $w(\sigma)>0$ and such that \blue{a (C2)-reduction is applied to $x_1$.}
Then by induction, there must be a phylogenetic network $\cN'$ on $\{x_1,\ldots, x_m\}$  that 
displays $\cF[2]$ and $\cF'[2]$
with $w((x_1,\dots,x_m))\ge r(\cN')$. But
since a (C2)-reduction corresponds to 
removing an edge from one of $\cF$ or $\cF'$, by Lemma~\ref{reviewer2} it 
follows that there is a phylogenetic network $\cN$ \blue{on $X$} that displays $\cF$ and $\cF'$ with 
$$
r(\cN) = r(\cN')+1 \ge w(\sigma')+ 1 = w(\sigma).
$$
This completes the proof of the theorem. }
\end{proof}

\teal{We end this section by noting that the proof of Theorem~\ref{uppercherry} in combination with the constructive proofs of Lemmas~\ref{reviewer1} and~\ref{reviewer2} lend themselves to a polynomial-time algorithm that reconstructs a phylogenetic network $\cN$ from a given cherry picking sequence $\sigma$ such that $w(\sigma)\geq r(\cN)$.}

\section{\blue{Four} technical lemmas}
\label{sect:tech}

In this section, we shall prove \blue{four} technical  lemmas 
\blue{(Lemmas~\ref{lem:simplify}, \ref{leafreduction}, \ref{reduction} and~\ref{lem:smallblobs})} 
that are concerned
	with how we can reduce the reticulation number of a phylogenetic network displaying two forests
	whilst \blue{retaining} the displaying property. We shall use these lemmas 
to prove Theorem~\ref{uppernet}  in the next section. 

We start by introducing some further definitions.
A {\em pseudo-network $\cN$ on $X$} is a connected multi-graph 
with leaf set $X \subseteq V(\cN)$ such that 
all non-leaf vertices have degree \blue{three}. Note that any multi-edge 
in a pseudo-network can contain only two edges, and that 
any pseudo-network  that does not contain a multi-edge is a phylogenetic network.
\blue{Note also that the suppression of a multi-edge and the subsequent suppression of the resulting degree-two vertices may create a new multi-edge.}
In addition, we extend the notions introduced for phylogenetic networks 
to pseudo-networks in the  natural way. 

\blue{Now, let $\cN$ be a pseudo-network on $X$. \teal{A maximal 2-connected\footnote{A graph $G$
is {\it 2-connected} if the graph obtained by removing any vertex from $G$ is connected. }
subgraph of $\cN$ that is not an edge is called a {\em blob} 
of $\cN$. In particular, note that a multi-edge in $\cN$ could be a blob. } 
Let $\cB$ be a blob of $\cN$, and let $e=\{u,v\}$ be an edge of $\cN$. We say that
$\cB$ is {\em incident} with $e$ (or that $e$ is {\em incident} with $\cB$) if precisely one of $u$ and $v$
is contained in $V(\cB)$. Furthermore, if $e$ is incident with $\cB$ and contains a leaf $x\in X$,
then we say that $\cB$ is {\em incident} with $x$.
We denote the subset of leaves of $\cN$
that are incident with $\cB$ by $L(\cB)$. Now let $E$ be the subset of edges of $\cN$ that are incident with $\cB$. We refer to $\cB$ as a {\it pendant blob} of $\cN$ if at most one edge in $E$ is a non-trivial cut-edge of $\cN$ and all other edges in $E$ are incident with a leaf.}


\teal{In what follows, for a pseudo-network $\cN$ on $X$ that displays two forests $\cF$ and $\cF'$, 
we shall be interested in understanding 
how $r(\cN)$ changes if we transform $\cN$ into a phylogenetic network that displays $\cF$ and $\cF'$ by removing certain 
leaves or edges. To this end, suppose that $|X|\geq 2$, and $\cN$ contains precisely one multi-edge}. 
We shall call a sequence $\cN_1,\ldots, \cN_k$,
$k\geq 2$, of distinct pseudo-networks on $X$ a \blue{{\em simplification sequence for $\cN$} if $\cN=\cN_1$,}
\blue{each} $\cN_{i+1}$ is obtained from $\cN_i$ by suppressing a single multi-edge in $\cN_i$ and suppressing the 
resulting vertices of degree two for all $1 \le i \le k-1$, and $\cN_k$ 
is a phylogenetic network on $X$.

%

\begin{lemma}\label{lem:simplify} 
	Suppose \blue{that} $\cF$ and $\cF'$ are two forests on $X$ \blue{with} $|X|\ge 2$
	and that $\cN$ is a pseudo-network on $X$ with precisely one  multi-edge 
	\blue{such that $\cN$} displays $\cF$ and $\cF'$.
	In addition, suppose that the blob in $\cN$ that 
	contains the multi-edge is incident with at least two cut-edges of $\cN$.
	Then there exists \blue{a simplification sequence} $\cN_1,\ldots, \cN_k$ for $\cN$, \teal{for} \blue{some $k\geq 2$}.
	Moreover, $\cN_k$ displays $\cF$ and $\cF'$ and $r(\cN) >r(\cN_k)$.
\end{lemma}
\begin{proof}
\teal{Let $\cB$ be the blob in $\cN$ that contains the unique multi-edge in $\cN$, and 
 let $u$ and $v$ be the two vertices contained in the multi-edge. In addition let $p$ and $q$ be the 
 vertices in $\cN$ such that $\{u,p\}$ and $\{v,q\}$ are in $E(\cN)$.
 Then since $\cB$ is incident with at least two cut-edges it 
 follows that $p \neq q$. Thus, the pseudo-network obtained 
 by suppressing the multi-edge and any resulting degree-two vertex
is either a phylogenetic network on $X$, or contains a blob  
that is incident with at least two cut-edges and contains a unique multi-edge.}
	This immediately implies that there is a unique simplification sequence 
	$\cN_1,\ldots, \cN_k$ for $\cN$, \teal{for} \blue{some $k\geq 2$}.
		
	We now show that $\cN_{i}$ displays the forests $\cF$ and $\cF'$, for all $1 \le i \le k$.
	\blue{For some $i\in\{1,\ldots, k-1\}$,} suppose \blue{that} $\cN_{i}$ displays the forests $\cF$ and $\cF'$, \blue{that $e=\{u,v\}$ and $e'=\{u,v\}$ are edges of $\cN_i$}
	and that, \blue{$e$ is suppressed}
	to obtain $\cN_{i+1}$.  Choose images for $\cF$ and $\cF'$ in $\cN$, respectively.
	It is straight-forward to check that if $e$ 
	is not contained in the image of $\cF$ or $\cF'$, then the 
	pseudo-network $\cN_{i+1}$ displays $\cF$ and $\cF'$.
	Moreover, if \blue{one of $e$ and $e'$, say $e$,}
	is contained in the image of $\cF$ and \blue{$e'$ is contained} in the image of
	$\cF'$, then clearly we can alter the 
	image of $\cF$ so that only \blue{$e'$ is} contained in the image of $\cF$ and $\cF'$.
	Since $\cN$ displays  $\cF$ and $\cF'$, it follows that $\cN_{i+1}$ also does for all $1 \le i \le k-1$.

	Finally, we note that $r(\cN) > r(\cN_k)$ since clearly $r(\cN_i) = r(\cN_{i+1})+1$, for all $1 \le i \le k-1$. 
\end{proof}

We \blue{now} consider the effect of removing certain leaves from a phylogenetic network 
that displays two forests. 
\blue{To this end, we define for a phylogenetic network $\cN$ on $X$ and $x \in X$, the graph $\cN-x$ to be the 
pseudo-network on \blue{$X\setminus\{x\}$} obtained from $\cN$ by removing $x$, 
the pendant edge $e_x \in  E(\cN)$, and suppressing the resulting degree-two vertex.} We first make a simple observation
whose proof is straight-forward.

\begin{observation}\label{help}
	Suppose \blue{that} $\cN=(V,E)$ is a phylogenetic network on $X$ \blue{with} $|X|\ge 2$, \blue{
	$x \in X$, and $\{u,x\}$, $\{u,v\}, \{u,w\} \in E$ with $u,v,w \in V\setminus\{x\}$. Then $\cN-x$ is a pseudo-network, and $\cN-x$ is a phylogenetic network on $X\setminus\{x\}$ if and only if $\{v,w\}\not\in E$.}
\end{observation}

\begin{figure}[t]
	\center
		\scalebox{1}{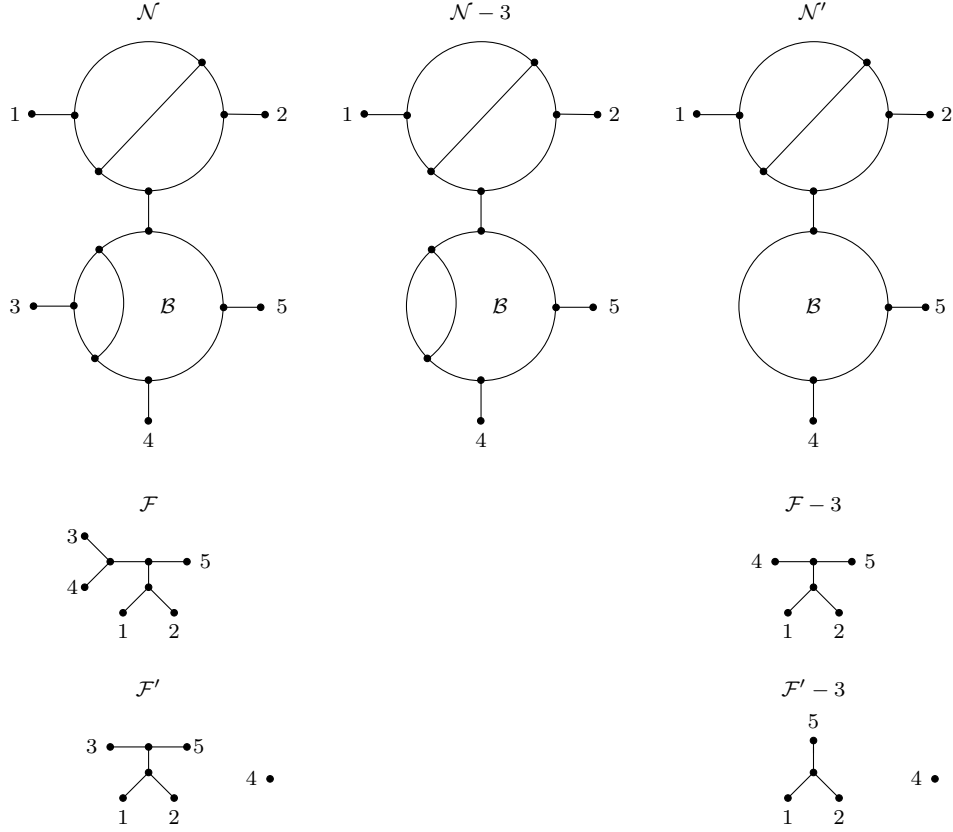}
	\caption{\blue{An explicit example of Lemma~\ref{leafreduction} applied to $x=3$ of  a phylogenetic network $\cN$ on $\{1,2,\ldots,5\}$ with a pendant blob $\cB$ that displays two forests $\cF$ and $\cF'$ on $\{1,2\ldots,5\}$. Then $\cN-3$ is not a phylogenetic network, and there exists a phylogenetic network $\cN'$ on $\{1,2,4,5\}$ with $r(\cN)>r(\cN')$ that displays $\cF-3$ and $\cF'-3$.}
	\label{fig:configuration}
 }
\end{figure}

\blue{Using Observation~\ref{help}, \teal{we obtain the following lemma, which is illustrated in Fig.~\ref{fig:configuration}.}}

\begin{lemma}\label{leafreduction}
	Suppose \blue{that} $\cF$ and $\cF'$ are two forests on $X$ \blue{with} $|X|\ge2$ 
	and \blue{that} $\cN$ is a phylogenetic network on $X$ that displays $\cF$ and $\cF'$.
	In addition, suppose that there is a pendant blob $\cB$ of $\cN$ with $|L(\cB)| \ge 2$
	and some $x \in L(\cB)$ such that the pseudo-network \blue{$\cN-x$ on $X\setminus\{x\}$} 
	is {\em not} a phylogenetic network. Then there exists a phylogenetic network
	$\cN'$ on \blue{$X\setminus\{x\}$} such that $r(\cN)>r(\cN')$ and $\cN'$ displays $\cF-x$ and $\cF'-x$.
\end{lemma}
\begin{proof}
Suppose \blue{first that $X=\{x,y\}$}. Then we can take $\cN'$ to be the vertex $y$ to complete the proof of the lemma.
	
Now, \blue{suppose that $|X|\ge 3$. Note that $\cN-x$ clearly displays $\cF-x$ and $\cF'-x$. Let $v$ and $w$ be the two vertices adjacent to the non-leaf vertex of $e_x$}. Since $\cN$ is a phylogenetic network, we have $v\ne w$. Moreover,
by Observation~\ref{help}, $\{v,w\}$ is an edge of $\cN$ and $\cN-x$ contains the unique multi-edge between $v$ and $w$. \blue{This multi-edge} is contained in a blob $\cB'$ of \blue{$\cN-x$. Since $\cB$ is a pendant blob} in $\cN$, it follows that $\cB'$ is \blue{a pendant blob in $\cN-x$. Furthermore, since $|L(\cB)|\geq 2$ and $|X|\ge 3$, either $|L(\cB')|\geq 2$, or $|L(\cB')|=1$ and $\cB'$ is incident with at least one non-trivial cut-edge.} 
The proof of the lemma now follows immediately by applying Lemma~\ref{lem:simplify} to \blue{$\cN-x$} and $\cB'$.
\end{proof}

We now consider what happens when we remove certain edges from a phylogenetic network 
that displays two forests.

\begin{lemma}\label{reduction}
	Suppose \blue{that} $\cF$ and $\cF'$ are two forests on $X$ \blue{with} $|X|\ge2$ 
	and \blue{that} $\cN$ is a phylogenetic network on $X$ that displays $\cF$ and $\cF'$. 
	Let $\cN[\cF]$ and $\cN[\cF']$ denote \blue{images} of $\cF$ and $\cF'$
	in $\cN$, respectively.	If $e$ is an edge in a pendant blob $\cB$ of $\cN$ with $|L(\cB)|\ge 2$ and
	\begin{itemize}	
	\item[(a)] $e$ is neither contained in  $\cN[\cF]$ nor  in $\cN[\cF']$, then 
	there is a phylogenetic network $\cN'$ on $X$ that displays $\cF$ and $\cF'$ with $r(\cN) > r(\cN')$.
	\item[(b)] 
 \blue{$e$ is not in $\cN[\cF']$ but there exists an edge $f$ in $\cF$ such that $e$ is in the image of $f$ in $\cN[\cF]$,}
   then there is a phylogenetic network 
	$\cN'$ on $X$ that displays $\cF -f$ and $\cF'$ and $r(\cN) > r(\cN')$.
   \end{itemize}
\end{lemma}
\begin{proof} 
Let $e$ be as in the statement of the lemma. 
Remove $e$ from $\cB$ and suppress the two 
resulting degree-two vertices. Then it is straight-forward
to check that the resulting multi-graph $\cN'$ is a 
pseudo-network on $X$ with $r(\cN) > r(\cN')$ that 
displays $\cF$ and $\cF'$ in case (a) and $\cF -f$ and $\cF'$ in case (b).

If $\cN'$ is a phylogenetic network, this \blue{completes} the proof of the lemma. Otherwise, 
since by assumption $\cB$ must be incident with at least two cut-edges, 
it follows that we can apply Lemma~\ref{lem:simplify} to complete the proof.
\end{proof}

\blue{Our final lemma in this section can be used to remove two special types of pendant blobs from a network.
Before stating the lemma, we first need a new definition.
Suppose that $\cN$ is a phylogenetic network on $X$, $|X|\ge 1$, 
	and that $\cB$ is a \blue{pendant} blob of $\cN$ with $|L(\cB)| \le 1$. 
	If $\cN \neq \cB$, then let $\{u,v\}$ be the cut-edge that 
	disconnects $\cB$ from $\cN$ with $v \in V(\cB)$.
	We define \blue{the phylogenetic network} $\cN-\cB$ as follows: }

\begin{itemize}	
\item  \blue{Suppose that $|L(\cB)|=0$, so
	that $\cN \neq \cB$ as $\cN$ has a non-empty leaf
	set. Then first delete all vertices of $\cB$ and any 
	edges containing them to obtain a graph $\cG$.
	If the pseudo-network obtained by then suppressing $u$ in $\cG$ is a 
	phylogenetic network, then this is defined to be $\cN-\cB$.
	Otherwise, suppressing $u$ in $\cG$ must result in a multi-edge 
	that contains the two vertices of $\cG$ adjacent with $u$, say $w,w'$.
	To obtain $\cN-\cB$ we 
	subdivide the edge $\{w,w'\}$ in the graph $\cG$ by adding
	in a new vertex $u'$ and then add the edge $\{u,u'\}$.}

\item \orange{Suppose $|L(\cB)|=1$, so that $L(\cB)=\{x\}$, \teal{for} some $x \in X$.}
\blue{If $\cN = \cB$, then define $\cN - \cB = \{x\}$. 
Otherwise, we let $\cN-\cB$ be the phylogenetic network obtained 
by first removing all vertices in $\cB$ and edges that contain them from $\cN$ 
and then adding the edge $\{x,u\}$ (i.e. the graph obtained by replacing $\cB$ with leaf $x$).}
\end{itemize}

\newpage

 \begin{lemma}\label{lem:smallblobs}
 	\blue{Suppose that $\cN$ displays two forests $\cF$ and $\cF'$ and
 		that $\cN$ is a phylogenetic network that contains a pendant blob $\cB$ 
 		with $|L(\cB)| \le 1$. Then $\cN-\cB$ displays $\cF$ and $\cF'$ and
 		 $r(\cN)> r(\cN-\cB)$.}
 \end{lemma}
 
 \begin{proof}
 \blue{It is straight-forward to see that 
$\cN-\cB$ displays $\cF$ and $\cF'$. }

\blue{Now consider the blob $\cB$. We first show that $r(\cB) \ge 2$. Towards a contradiction, assume that
$r(\cB) < 2$.
Then $r(\cB) = 1$, since $r(\cB)\neq 0$ as otherwise $\cB$ 
would be a tree which is impossible as it is 2-connected. 
\blue{Consider} the graph $\cB'$ that
is obtained \blue{from $\cB$} by suppressing all vertices in $\cB$ with degree 2. 
Then $\cB'$ is 2-connected and every vertex in $\cB'$ must 
have degree 3. Moreover, since $r(\cB')=r(\cB)=1$ 
there must be an edge whose removal from $\cB'$ gives a tree, since
a connected graph $G$ with $r(G)=1$ contains a unique cycle \cite[Corollary 2, p.29]{berge01}. 
But, if we remove an edge from $\cB'$ 
then every vertex in the resulting graph
must have degree at least 2, and in a tree that is not a single vertex 
there must be at least two vertices with degree 1, a contradiction. So $r(\cB)\ge 2$.}

\blue{Next, suppose $|L(\cB)|=0$. Then
	if the pseudo-network obtained by suppressing $u$ in $\cG$ 
	in the definition of $\cN-\cB$ is a 
	phylogenetic network, then it is
	straight-forward to check that $r(\cN-\cB)=r(\cN)-r(\cB)$, and so
	$r(\cN)> r(\cN-\cB)$. Otherwise, it is straight-forward to 
check that $r(\cN-\cB) = r(\cN)-r(\cB)+1$. Hence, as $r(\cB) \ge 2$,
it follows that $r(\cN)> r(\cN-\cB)$. }

\blue{And, finally, 
it is straight-forward to check that if $|L(\cB)|=1$, 
then $r(\cN-\cB)=r(\cN)-r(\cB)$ and so $r(\cN)> r(\cN-\cB)$.}
\end{proof}

\section{Bounding the weight of a cherry picking sequence from above by the hybrid number}
\label{sect:bound2}

In this section, we complete the proof of Theorem~\ref{main} by establishing the following result.

\begin{theorem} \label{uppernet}
	Suppose that $\cN$ is a phylogenetic network on $X$ that displays two forests $\cF$ and $\cF'$ on $X$. Then 
	there is a cherry picking sequence $\sigma$ for $\cF$ and $\cF'$ with $r(\cN) \ge w(\sigma)$.
\end{theorem}

\blue{Before proving this result, we first state 
a useful lemma that will help us to deal with (C1) and (C3) reductions.
Its proof is straight-forward and left to the reader.}
	
\begin{lemma}\label{reviewer3}
\blue{Suppose that $\cN$ is a phylogenetic network on
$X$, $|X|\ge 2$, and that $\cF$ and $\cF'$ are forests on $X$ that are displayed by $\cN$.
If $x \in X$ is such that $\cN-x$ is a phylogenetic network
then $r(\cN-x)=r(\cN)$, and $\cN-x$ displays both $\cF-x$ and $\cF'-x$.
Moreover, if $\cN$ has a cherry, then the cherry contains a leaf $x$ 
such that $\cN-x$ is a phylogenetic network and either (a) there is a cherry $(x,y)$ contained
in both $\cF$ and $\cF'$, $y \in X\setminus\{x\}$, or (b) $x$ is
a component in one of  $\cF$ or $\cF'$.}
\end{lemma}

\orange{Before proving Theorem~\ref{uppernet}, as the proof is 
quite technical, we give a brief road-map of its main points.
The proof works using induction on $r(\cN)$. In the base case
we have $r(\cN)=0$, and so $\cN$ is a phylogenetic tree. 
In this case we can continually pick off leaves that are 
contained in cherries to obtain a cherry picking sequence $\sigma$ with 
$w(\sigma)= r(\cN) =0$.}

\orange{Now, in case $r(\cN)>0$, we essentially repeatedly 
look for non cut-edges $e$ in $\cN$ that are contained in 
the image of one of the forests displayed by $\cN$ but not in
the other. If we can find such an edge $e$, by carefully
considering how the images of the forests are 
embedded in $\cN$, we can then complete the proof by 
removing $e$ from $\cN$, applying Lemmas~\ref{leafreduction} and 
\ref{reduction} and some (C2)-reduction, and 
using the fact that $r(\cN) > r(\cN-e)$. We use
this tool within a series of claims to show that either
we can prove the theorem by induction, or that we can 
further restrict the way that the 
forests are embedded in $\cN$. This eventually
allows us to finally apply our induction hypothesis one final time and thus complete the proof.}

\orange{More specifically, \teal{we first show that
	we can assume that $\cN$ satisfies a certain property (P).
	Then, after noting that} $\cN$ must be pendantless (Claim~1), we 
show that we can also assume that $\cN$ must have
a pendant blob $\cB$ with at least two leaves 
such that no leaf \teal{of} $\cB$ is a component of $\cF$ or $\cF'$ 
and every edge in $\cB$ is in $\cN[\cF]$ or $\cN[\cF']$ (Claims~2, 3 and 4).
In Claims~5 and 6 we then show that either $\cF$ or $\cF'$
must contain a cherry $(x,y)$ \teal{that, out of the four possibilities (i)--(iv) in Observation 3.1,
only satisfies (ii).}
Then in Claims~7--11 we show that we can assume that 
there is a component in $\cF$ that contains a cherry $(x,y)$
and a component in $\cF'$ that contains a cherry $(x,z)$ with $x,y,z \in L(\cB)$.
This allows us to assume that $((x,y),p)$ or $((x,y),(p,q))$ is a pendant subtree of one of 
the forests for some $x,y,p,q \in L(\cB)$ distinct. This permits us to carry 
out a case analysis in Claims~12--14 in order to complete the proof.}\\

\noindent{\blue{\bf Proof of Theorem~\ref{uppernet}}:}
\blue{In what follows, for two sequences 
$\sigma=(x_1,\ldots, x_m)$, $m\geq 1$, and $\sigma'=(y_1,\ldots, y_l)$, $l\geq 1$, 
we denote  by  $\sigma\circ\sigma'$ the sequence $(x_1,\ldots, x_m,y_1,\ldots, y_l)$.}

\blue{We prove the theorem using induction on $r(\cN) \ge 0$.}

\noindent \blue{Base Case: Suppose $r(\cN)=0$. Then $\cN$ is a 
phylogenetic tree on $X$. If $|X|=1$, then $X=\{x\}$, \teal{for}  some $x$, and
$(x)$ is a cherry picking sequence  for $\cF$ and $\cF'$ with $w((x))=0$.  
Now, suppose $m=|X|>1$. Order the elements in $X$ as 
$x_1,x_2,\dots, x_m$ so that, for all $1 \le i \le m-1$, $x_i$ is contained
in a cherry of $\cN|{\{x_i,\dots,x_m\}}$ which is possible since any phylogenetic tree 
with two or more leaves always contains a cherry.
Then, by applying  Lemma~\ref{reviewer3} 
and using either a (C1) or (C3) reduction for each $x_i$, $1 \le i \le m-1$, 
depending on whether (a) or (b) holds \teal{in Lemma~\ref{reviewer3}} for $x_i$, respectively, it follows
that $(x_1,\dots,x_m)$ is 
a cherry picking sequence for $\cF$ and $\cF'$ with 
$r(\cN)= w((x_1,\dots,x_m))=0$.
Hence the base case holds.}

\blue{Now, suppose that $q$ is a positive integer 
and that the theorem holds for all phylogenetic networks with reticulation 
number strictly less than $q$. Suppose that $\cN$ is a
phylogenetic network with $r(\cN)=q$ that displays $\cF$ and $\cF'$. Then 
to complete the proof of the theorem we need  to show that:}\\
  
\noindent {\bf (CP)}  There is a cherry picking sequence $\sigma$ for $\cF$ and $\cF'$ with $r(\cN) \ge w(\sigma)$.\\

\noindent \blue{Note first that we \teal{will assume} from now on that $|X|\geq 3$ since for
\teal{$|X|\in\{1,2\}$}, we can clearly find a cherry picking sequence 
$\sigma$ such that $w(\sigma)=0$ and so $r(\cN)>0=w(\sigma)$ and  hence (CP) holds.}

Starting with $\cN$, repeatedly look for elements $x \in X$
that satisfy (a) there is a cherry $(x,y)$ contained
in both $\cF$ and $\cF'$, $y \in X\setminus \{x\}$, or (b) $x$ is
a component in either $\cF$ or $\cF'$, and such that removing $x$ and $e_x$
and suppressing the resulting degree-2 vertex results
in a phylogenetic network 
until this is no longer possible. Then we
obtain a sequence $(x_1,\dots,x_m)$, $x_i \in X$, $m \ge 1$, where,
for $\cN_1=\cN$, $\cN_{i+1} = \cN_{i}-x_i$, $\cF[1]=\cF$, $\cF[i+1] = \cF[i]-x_i$, 
$\cF'[1]=\cF'$ and $\cF'[i+1] = \cF'[i]-x_i$, $1 \le i \le m$, 
we have for all $1 \le i \le m$ that 
$x_i$ satisfies  (a) or (b)  and $\cN_{i+1} = \cN_i-x_i$ is 
a phylogenetic network on $X\setminus\{x_1,\dots,x_i\}$, 
$\cN_i$ displays the forests $\cF[i]$, $\cF'[i]$ and, by Lemma~\ref{reviewer3}, $r(\cN_i)=r(\cN)$.

In particular, $\cN_{m+1}$ is a phylogenetic network on 
$X\setminus\{x_1,\dots,x_m\}$ with $r(\cN_{\blue{m+1}})=q$
that displays $\cF[m+1]$ and $\cF'[m+1]$,
and there is no $x \in X\setminus\{x_1,\dots,x_m\}$ such that $\cN_{m+1}-x$ is a
phylogenetic network and either (a) or (b) holds.

\teal{Now suppose that we could show that (CP) holds for 
	$\cN_{m+1}$ along with the forests $\cF[m+1]$ and $\cF'[m+1]$ (which we shall do below).} 
Then there must exist a cherry picking sequence $\sigma'$ for $\cF[m+1]$ and 
$\cF'[m+1]$ with $r(\cN_{m+1}) \ge w(\sigma')$.  Hence, $\sigma = (x_1,\dots,x_m) \circ \sigma'$ 	
is a cherry picking sequence for $\cF$ and $\cF'$. 
Since we have applied a (C1) or (C3) reduction to each $x_i$
depending on whether (a) or (b) holds for $x_i$, respectively, 
we obtain $w(\sigma')= w(\sigma)$.
Thus $r(\cN) = r(\cN_{m+1}) \ge w(\sigma')= w(\sigma)$
and so (CP) holds for $\cN$. 

\teal{Thus, to show that (CP) holds for $\cN$, we assume from now on that $\cN$ satisfies the following property:} \\

\noindent\blue{{\bf (P)} There
is no $x \in X$ such that $\cN-x$ is a phylogenetic network on $X \setminus \{x\}$ and
either (a) there is a cherry $(x,y)$ contained
in both $\cF$ and $\cF'$, $y \in X\setminus \{x\}$, or (b) $x$ is
a component in either $\cF$ or $\cF'$.} \\

\noindent {\bf Claim~1:} $\cN$ is pendantless.\\ 

\noindent {\bf Proof of Claim~1:}
Suppose $\cN$ is {\em not} pendantless. Then $\cN$ contains a cherry. 
Thus by Lemma~\ref{reviewer3} there must be some $x \in X$
such that $\cN-x$ is a phylogenetic
network and (a) or (b) is satisfied for $x$. But this contradicts
\teal{the fact that $\cN$ satisfies property (P)}.
$\blacksquare$\\

\blue{By Claim~1 there is some pendant blob in $\cN$. Let $\cB$ be such a blob.}\\

\noindent\blue{{\bf Claim~2:} $|L(\cB)|\ge 2$.} \\

\noindent\blue{{\bf Proof of Claim~2:}
If $|L(\cB)|\le 1$, then
by applying Lemma~\ref{lem:smallblobs} to $\cN$ and using induction it 
follows that (CP) holds. $\blacksquare$}\\

\noindent\blue{{\bf Claim~3:} \teal{No} leaf in $L(\cB)$ is a component in $\cF$ or $\cF'$.}\\

\noindent\blue{{\bf Proof of Claim~3:}
Suppose that some $x \in L(\cB)$ is a component of $\cF$.
Then by \teal{property (P)}, the pseudo-network $\cN-x$ is not a phylogenetic network. 
So we can apply a (C3)-reduction to $x$ and, since $|L(\cB)|\ge 2$ by Claim~2, we can also 
apply Lemma~\ref{leafreduction} to $\cF$, $\cF'$, 
$\cN$ and $x$ to see that (CP) must hold for $\cN$ by induction}. 
$\blacksquare$\\

In what follows, we choose images $\cN[\cF]$ 
and $\cN[\cF']$ of $\cF$ and $\cF'$ in $\cN$, respectively, \blue{as} 
defined in Section~\ref{sect:prelim}.\\


\noindent \blue{{\bf Claim 4:} \teal{Every} pendant edge incident to $\cB$
	and every edge in $\cB$ is contained in $\cN[\cF]$ or $\cN[\cF']$ (or both).}\\

\noindent{{\bf Proof of Claim~4:}} If a pendant edge were
not in $\cN[\cF]$ or $\cN[\cF']$, then either 
$\cF$ or $\cF'$ would contain an element in
$L(\cB)$ that is a component in $\cF$ or $\cF'$, which contradicts our 
assumption \blue{in Claim~3}.
And, if there is some 
$e \in E(\cB)$ such that $e$ is neither contained in $\cN[\cF]$ 
nor in $\cN[\cF']$,
then \blue{since $|L(\cB)|\ge 2$ by Claim~2,} 
Lemma~\ref{reduction}(a) \blue{implies that} there exists a phylogenetic network $\cN'$ on $X$
that displays $\cF$ and $\cF'$ and $r(\cN)>r(\cN')$. 
\blue{We can then apply induction to see that (CP) holds for $\cN$.} $\blacksquare$\\

%

\noindent\orange{{\bf Claim 5:} Suppose that $(x,y)$ is a cherry in a 
component $\cT$ of $\cF$ or $\cF'$ with $x,y \in  L(\cB)$, and 
$\gamma$ is the path in $\cT$ between $x$ and $y$.
If $u,v$ are the vertices in $\cB$ adjacent to $x,y$, respectively,
then $\{x,u\}$ and $\{y,v\}$ are pendant edges of $\cB$,  
the image $\cN[\gamma]$ of $\gamma$ contains both $\{x,u\}$ and $\{y,v\}$ and
has length at least three, and every edge in $\cN[\gamma]$
not equal to $\{x,u\}$ and $\{y,v\}$ is contained in $E(\cB)$.}\\

\noindent\orange{{\bf Proof of Claim~5:}
Since $\cN$ is pendantless by Claim~1, $u \neq v$ and 
$\{x,u\}$ and $\{y,v\}$ are pendant edges of $\cB$.
Moreover, $u$ and $v$
are clearly contained in $\cN[\gamma]$, and $\cN[\gamma]$ must contain at least
three edges since $u\neq v$. And
every edge in $\cN[\gamma]$
not equal to $\{x,u\}$ and $\{y,v\}$ is contained in $E(\cB)$ 
since $\cB$ is a pendant blob in $\cN$. $\blacksquare$} \\

\noindent \blue{{\bf Claim 6:} If $\mathcal T$ is any tree in $\cF$ or in $\cF'$
that contains a cherry $(x,y)$ with $x, y \in L(\cB)$,
then 
out of the four cases (i)--(iv) in Observation~\ref{note} 
only case (ii) holds for $x$ and $y$.}\\

\noindent{{\bf Proof of Claim~6:} }
\blue{Assume without loss of generality that $\cT$ is a component in $\cF$
and that $(x,y)$ is a cherry in $\cT$.}
\orange{Let $u,v,\gamma$ and $\cN[\gamma]$ be as in Claim~5.} 

Suppose Observation~\ref{note}(iv) holds, i.e. $x$ and $y$  are contained in different components in $\cF'$
and neither $x$ nor $y$ is in a cherry of $\cF'$.
Then as $x$ and $y$ are in different components of $\cF'$, there must exist an edge $e'$ in 
$\cN[\gamma]$ that is not in $\cN[\cF']$. 
So, by \blue{Claim 2 and} applying Lemma~\ref{reduction}(b) to $e'$, there 
exists a phylogenetic network $\cN'$ on $X$
with $r(\cN)>r(\cN')$ that displays $\cF'$ and
either $\cF-e_x$ or $\cF-e_y$, say $\cF-e_x$.
Since $x, y \in L(\cB)$ neither $x$ nor $y$ is an isolated vertex in $\cF'$ by Claim 3, 
and so we have applied a (C2)(c)-reduction to $\cF$ and $\cF'$ to obtain $\cF-e_x$. Moreover, by induction, 
there is a cherry picking sequence $\sigma'$ for
$\cF-e_x$ and  $\cF'$ with $r(\cN') \ge w(\sigma')$. Hence, $\sigma=(x)\circ\sigma'$ is 
a cherry picking sequence for $\cF$ and $\cF'$ such that
$r(\cN) \ge r(\cN')+1 \ge w(\sigma') + 1 = w(\sigma)$, and so 
\blue{(CP) must hold.}

Now, suppose Observation~\ref{note}(iii) holds,
i.e, $x$ and $y$ are leaves in the same tree $\cT'$ in $\cF'$, but neither 
$x$ nor $y$ is in a cherry of $\cT'$.
If there exists an edge in $\cN[\gamma]$ that is not contained in $\cN[\cF']$ then
we can apply a similar argument to the one used in the previous paragraph
\blue{to see that (CP) must hold for $\cN$}
using a (C2)(b)\blue{(i)}-reduction \blue{applied to $x$ or $y$} and induction.
So, suppose every edge in $\cN[\gamma]$ is contained in $\cN[\cF']$. 
Let $e$ and $e'$ be the edges in $E(\cB)$ incident 
to $u$ and $v$, respectively, but not in $\cN[\gamma]$
(so, in particular, neither $e$ nor $e'$ is a cut-edge).
Note that $e \neq e'$ since if $e=e'$, then
as every edge in $\cB$ is contained in
$\cN[\cF]$ or $\cN[\cF']$ \blue{by Claim~4}, this would imply that
either $\cF$ or $\cF'$ contains a cycle, a contradiction.
Moreover, since every edge in $\cB$ is contained in
$\cN[\cF]$ or $\cN[\cF']$ \blue{by Claim~4} and $(x,y)$ is a cherry in $\cF$,
it follows that one of the edges $e,e'$, say $e'$, 
is in $\cN[\cF']$ but not in $\cN[\cF]$, and that \blue{$e'$} is in 
the image $\cN[f]$ of the edge $f$ in \blue{$\cT'$ 
that contains the vertex in $\cT'$ that is adjacent to $y$ but that
is not in the path in $\cT'$ between $x$ and $y$}. 
So, by Lemma~\ref{reduction}(b) applied to $e'$ with the roles of $\cF$ and $\cF'$ reversed, there exists a phylogenetic network
$\cN'$ that displays $\cF$ and $\cF'-f$ with $r(\cN)>r(\cN')$.
Since this corresponds to applying a (C2)(b)\blue{(ii)}-reduction to $f$, we can apply induction again 
to $\cN'$ \blue{to see that (CP) must hold.}

\blue{Last, suppose Observation~\ref{note}(i) holds,
i.e., that $(x,y)$ is a cherry in both $\cF$ and $\cF'$. Note that
by \teal{property (P)}, neither the pseudo-network $\cN-x$ nor
the pseudo-network $\cN-y$ is a phylogenetic network.
Moreover, every edge in $\cN[\gamma]$ must be
contained in both $\cN[\cF]$ and $\cN[\cF']$, 
otherwise we could use  Lemma~\ref{reduction}(b)  and a (C2)(b)\blue{(i)}-reduction and induction
to see that (CP) must hold.
Let $e$ and $e'$ be the edges in $E(\cB)$ incident 
to $u$ and $v$, respectively, but not in $\cN[\gamma]$, 
so that $e\neq e'$ and neither $e$ nor $e'$ is a cut-edge of $\cN$. 
Note that since every edge in $\cB$ is contained
in $\cN[\cF]$ or $\cN[\cF']$ by Claim~4, and $(x,y)$ is a 
cherry in both $\cF$ and $\cF'$, it follows
without loss of generality, that $e$ is 
in $\cN[\cF]$ but not in $\cN[\cF']$, and that $e'$ is 
in $\cN[\cF']$ but not in $\cN[\cF]$.}

\blue{We now show that the length of the path $\cN[\gamma]$ must be three
(noting that we are assuming that it has length at least three).
Suppose not. First suppose that the length of $\cN[\gamma]$ is four.
Let $w$ be the vertex in  $\cN[\gamma]$ that is 
distance two from both $x$ and $y$. \teal{Then,
by Claim~4, $e$ and $e'$ must be in either $\cN[\cF]$ or $\cN[\cF']$.
Since $(x,y)$ is a cherry in both $\cF$ and $\cF'$, 
it follows that the edge incident to $\cB$ and not in $\cN[\gamma]$
is neither in $\cN[\cF]$ nor $\cN[\cF']$. Hence,
$w$ must be contained in the cut-edge of $\cN$ that is incident to $\cB$.}
Now, if $e$ and $e'$ have no vertex in common then, 
since $\cN-x$ is a pseudo-network that is not a phylogenetic network, 
when we remove the edge $\{u,x\}$ from $\cN$
and suppress the vertex $u$, we must obtain a multi-edge 
that contains the two vertices in $\cB$ adjacent to $u$ 
that are not equal to $x$. But this is impossible
since one of these two vertices must be equal to $w$, 
\teal{and the degree of $w$ in $\cN$ is 3}.
And, if $e$ and $e'$ have a vertex in common, say $w'$, then 
$\cB$ must be a 4-cycle, with leaves $x,y$ and
another leaf $z$, $z \in X$, such that $\{w',z\}$ is a
pendant edge of $\cB$. But this is impossible \teal{by property (P)}
since $(z,y)$ is then a cherry in both $\cF$ and $\cF'$ and
$\cN-z$ is a phylogenetic network.
Finally, note that since $(x,y)$ is cherry in both
$\cF$ and $\cF'$ and $\cB$ is a pendant blob, it
follows that the length of $\cN[\gamma]$ 
cannot be greater than four by Claim~4.}

\blue{Now, assume that the length of 
the path $\cN[\gamma]$ is three.
If $e$ and $e'$ have no vertex in common, then 
since $\cN-x$ is a pseudo-network that is not a phylogenetic network, 
when we remove the edge $\{u,x\}$ from $\cN$
and suppress the vertex $u$, we must obtain a multi-edge 
that contains the two vertices in $\cB$ adjacent to $u$ 
that are not equal to $x$. But this is impossible
since one of these two vertices must be equal to $v$ as 
$\cN[\gamma]$ has length three, and $e$ and $e'$ have 
no vertex in common. \teal{And, finally for the case where
$\cN[\gamma]$ has length three, if $e$ and $e'$ have a vertex
in common, say $c$, then $\cB$ is the 3-cycle $u,v,c$ and 
there is a vertex $c' \not= u,v$ with $c' \not\in X$  such that $\{c,c'\}$ is 
a non-pendant cut-edge of $\cN$.} So we can
apply a (C1)-reduction to $x$, make a new network $\cN'$ on $X\setminus\{x\}$ with $r(\cN)>r(\cN')$
by removing $x$ and all of the vertices in $\cB$ and the edges in $\cN$ containing
them and adding the edge $\{c',y\}$, and then apply
induction to see that (CP) holds.}
$\blacksquare$\\

\noindent \blue{{\bf Claim 7:} \teal{Neither} $\cF$ nor $\cF'$ contains a cherry in
the form of an isolated edge $\{x,y\}$ with $x,y \in L(\cB)$. }\\

\noindent\blue{{\bf Proof of Claim~7:} Suppose that $\cT$ is a cherry in $\cF$ 
in the form of the isolated edge $\{x,y\}$ with  $x,y \in L(\cB)$. 
\orange{Let $u,v,\gamma$ and $\cN[\gamma]$ be as in Claim~5.} 
By Claim~6, 
Observation~\ref{note}(ii) holds and so
we may assume that there is a cherry in $\cF'$, say $(x,z)$, 
with $z\in \blue{X\setminus\{y\}}$.
Now, all of the edges in $\cN[\gamma]$ must be
contained in $\cN[\cF']$, otherwise we
can apply Lemma~\ref{reduction}(b) to any edge on $\cN[\gamma]$ not contained in $\cN[\cF']$ and a (C2)(a)(i)-reduction to $x_i$ plus induction
to see that (CP) holds. But this implies that the 
edge $e$ in $\cB$ that contains the vertex adjacent to $x$ and
that is not contained in $\cN[\gamma]$
must be contained in the image of the path 
in $\cF'$ between $x$ and $z$. Indeed, if this were not the
case then, as $e$ must be contained in $\cN[\cF']$ by Claim~4, 
this would imply that $(x,z)$ is not a cherry in $\cF'$.
Thus, as $\{x,y\}$ is a component in $\cF$ and $e$ is not a cut-edge for $\cN$, 
we can apply Lemma~\ref{reduction}(b) to $e$ and the roles of $\cF$ and $\cF'$ reversed
to see that (CP) holds by applying a (C2)(a)(i)-reduction to $x$ and induction.}
$\blacksquare$\\

\noindent \blue{{\bf Claim 8:} \teal{There} exists some tree $\cT$ in $\cF$ 
that contains a cherry $(x,y)$ with $x,y \in L(\cB)$, and such that
$|L(\cB) \cap L(\cT)|\ge 3$.}\\

\noindent\blue{{\bf Proof of Claim 8:} First note that
there must be some cherry $(x,y)$ in $\cF$ with $x,y \in L(\cB)$.
\orange{Indeed, suppose first that $\cN=\cB$. Then, $L(\cB)=X$, and so, by Claim~3, $\cF$ must contain a tree $\cT$ with $|L(\cT)|>1$ and $L(\cT)\subseteq L(\cB)$ and so such a cherry must exist.}  So suppose $\cN \neq \cB$.
Note that if $\cF$ contains a tree $\cT$ with $L(\cB)\subseteq L(\cT)$, then since $|L(\cT )|  \ge  |L(\cB)|  \geq 2$, there must be a cherry  $(x, y)$ in $\cT$ with $x, y \in L(\cB)$, otherwise $\cN$ would not display $\cT$. And, if $\cF$ contains no tree $\cT$ with  $L(\cB) \subseteq L(\cT )$, then $\cF$ must contain  a tree $\cT'$ with $L(\cT')$ a proper subset of $L(\cB)$, otherwise  $\cN$ would again not display $\cF$. But then, $\cT'$ contains a cherry  $(x, y)$ with $x, y \in L(\cB)$.
}

\blue{Now, let $\cT$ be the tree in $\cF$ that contains 
a cherry $(x,y)$ with $x,y \in L(\cB)$.}
\blue{First note that, by Claim~7, $|L(\mathcal T)| \ge 3$. We now show that
we must also have $|L(\mathcal B)| \ge 3$. Suppose for
contradiction that this is not
the case. Then, by Claim~2, we must have $L(\mathcal B) = \{x,y\}$, and
since $|L(\mathcal T)| \ge 3$, we must also have $\mathcal N \neq \mathcal B$.
Moreover, by Claim~3, neither $x$ nor $y$ is a component of $\mathcal F'$.
Hence, $(x,y)$ must be a cherry in $\mathcal F'$ since
otherwise $\mathcal N$ would not display $\mathcal F'$. But this
is a contradiction since it implies that $(x,y)$ is
a cherry in both $\mathcal F$ and $\mathcal F'$ so that Observation 3.1(i) holds, which 
is not possible by Claim~6. Thus, $|L(\cB)|\geq 3$.}  

\blue{Now, if $\cN = \cB$, 
then since $|L(\cT)|\ge 3$ and $|L(\cB)|\ge 3$, it 
immediately follows that $|L(\cB) \cap L(\cT)|\ge 3$.
So, suppose $\cN \neq \cB$. Then as $|L(\cB)|\ge 3$ it follows
that there is some $s \in \blue{L(\cB)\setminus\{x,y\}}$. If $s \in L(\cT)$, then 
$|L(\cB) \cap L(\cT)|\ge 3$ and Claim~8 follows.
Otherwise $s\not\in L(\cT)$ and so there must be some tree $\cT'$ in $\cF$ with $s \in L(\cT')$.
Now, since $|L(\cT)|\ge 3$, it follows that 
$\cN[\cT]$ must contain the non-pendant cut-edge incident to $\cB$, 
since otherwise $|L(\cB) \cap L(\cT)|\ge 3$. Thus, as 
$\cN[\cT']$ is disjoint from $\cN[\cT]$ by the definition of the image of a forest, we must 
have $L(\cT') \subseteq L(\cB)$. 
But $\cT'$ is not a single leaf by Claim~3, and it is not a
cherry in the form of an isolated edge by Claim~7.
Hence $|L(\cT')|\ge 3$, and so $\cT'$ contains a cherry $(r,t)$ with $r,t\in L(\cB)$ and $|L(\cT') \cap L(\cB)|\ge 3$. $\blacksquare$}\\

\blue{Now, by Claim~8, we shall  
assume that we have a tree $\cT$ in $\cF$ that has a
cherry $(x,y)$ with $x,y \in L(\cB)$ and such that
$|L(\cB) \cap L(\cT)|\ge 3$. 
Moreover, by Claim~6, Observation~\ref{note}(ii) holds, 
and so, as well as the cherry $(x,y)$ in $\cT$, 
there is a cherry in $\cF'$, say $(x,z)$, 
with $z\in \blue{X\setminus\{y\}}$.
Let $\cT'$ denote the tree in $\cF'$ that contains 
the cherry $(x,z)$.}
\orange{For the cherry $(x,y)$ in $\cT$ and
$\gamma$ the path in $\cT$ between $x$ and $y$, let $u,v$ and $\cN[\gamma]$ 
be as in Claim~5.
In addition, let $\delta$ be the path in $\cF'$ between $x$ and $z$ 
and let $\cN[\delta]$ be its image
in $\cN[\cF']$, noting  that $\cN[\delta]$ must have length at least three by Claim~5.
We picture this configuration in Fig.~\ref{fig:configuration2}.}\\

\noindent\orange{{\bf Claim 9:} \teal{Every} edge in $\cN[\gamma]$ 
or in $\cN[\delta]$ that is 
in $E(\cB)$ must be contained in $\cN[\cF']$ or $\cN[\cF]$, respectively.
In particular, since $x,y \in L(\cB)$, $\cN[\gamma]$ is contained in $\cN[\cF']$.}\\

\noindent \orange{{\bf Proof of Claim 9: }
We consider $\cN[\gamma]$; the proof for $\cN[\delta]$ is similar.
Suppose $g$ is an edge that is in  $\cN[\gamma]$  and $E(\cB)$, but $g$ is not in $\cN[\cF']$.
Then we can apply a (C2)(a)(i)-reduction to $x$ or $y$, remove the edge $g$ and apply Lemma~\ref{reduction}(b) to the removed edge plus induction
to see that (CP) must hold. Thus, $\cN[\gamma]$ must be contained in $\cN[\cF']$, by Claim 3. $\blacksquare$}\\

\noindent\orange{{\bf Claim 10:} Let $e$ and $e'$ be the edges in $\cN$ that are not in
the path $\cN[\gamma]$ and that contain $u$ and $v$, respectively (see Fig.~\ref{fig:configuration2}).
Then  
$e$ is in $\cN[\cT]$, $e'$ is not in $\cN[\cT]$, and both $e$  and $e'$ are in $\cN[\cF']$.}\\

\noindent\blue{{\bf Proof of Claim~10:} 
Note that neither $e$ nor $e'$ is a cut-edge.
Moreover, we can assume that $e$ must be contained in $\cN[\cT]$. Indeed, 
if not then $e$ must be contained in $\cN[\cF']$, by
Claim~4. Thus, since $\cN[\gamma]$ is contained in $\cN[\cF']$ \orange{by Claim~9}, 
$e$ is an edge in $\cN[\delta]$. Thus, we can apply a (C2)(a)(i) 
reduction to $z$, remove $e$, and apply Lemma~\ref{reduction}(b) 
to $e$ plus induction to see that (CP) holds. Hence, in summary, 
since $(x,y)$ is a cherry in $\cT$, we must have that $e$ is an 
edge in $\cN[\cT]$ and that $e'$ is not an edge in $\cN[\cT]$.}

\blue{Now, since $e'$ is not an edge in $\cN[\cT]$, by Claim ~4 it follows that
$e'$ is in $\cN[\cF']$.
Moreover $e$ is also in $\cN[\cF']$, since otherwise
as $(x,z)$ is a cherry in $\cT'$ it follows
that $\cF'$ contains the component $((x,y),z)$ and so we
could apply a (C2)(a)(i)-reduction to $z$, remove $e$ and 
apply Lemma~\ref{reduction}(b) to $e$ plus induction to see that (CP) must hold. $\blacksquare$}\\

\begin{figure}[t]
	\center
		\scalebox{1.4}{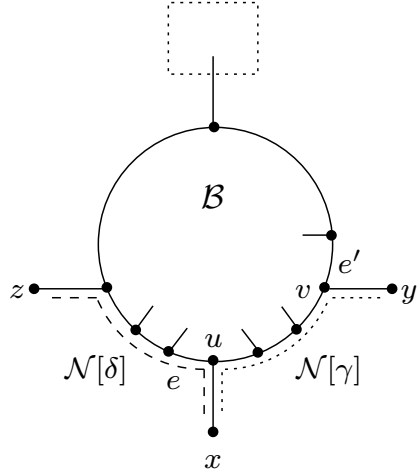}
	\caption{\orange{A schematic representation 
        illustrating the images $\cN[\gamma]$ and $\cN[\delta]$ of the cherries $(x,y)$ and
        $(x,z)$ in the blob $\cB$ as well as the associated edges $e$ and $e'$ used in the proof of Theorem~\ref{uppernet}. Note that we show that $z \in L(\cB)$ in Claim~11.}
	\label{fig:configuration2}
 }
\end{figure}

\noindent \orange{{\bf Claim~11:} $z \in L(\mathcal B)$,
every edge in $\cN[\delta]$ that does not contain $x$ or $z$ is in $E(\cB)$,
and  $\cN[\delta]$ is contained in $\cN[\cT]$.} \\

\noindent \orange{{\bf Proof of Claim~11:} Suppose that $z$ is not contained in  $L(\mathcal B)$. Then we
must have $\cN \neq \mathcal B$, and $\cN[\delta]$ must contain
a non-pendant cut-edge, such that  one of the vertices in this
cut-edge, $w$ say, is contained in $V(\cB)$.}

\orange{Consider the edge $e'$ (see Fig.~\ref{fig:configuration2}). 
Then since $e'$ is in $\cN[\cF']$ by Claim~10 and it is not a cut-edge,  
there must be a cherry in $\cT'$, $(r,s)$ say, with $r,s \in L(\cB) \setminus \{x,y\}$ (since otherwise $\cT'$ would only have four leaves and we could apply a (C2)(a)(i)-reduction and Lemma~\ref{reduction}(b) to $e'$).
But then by Claim~6, without loss of
generality $\cT$ must contain a cherry
$(r,a)$, with $a \in X \setminus \{x,y,s\}$. Moreover , since
$w$ is a vertex in $ \cN[\cT]$ we must have $a \in L(\cB)$ as every edge in $E(\cB)$ and $\cN[\delta]$ is contained in $\cN[\cT]$ by 
Claim 9.
Hence, instead of considering the cherries $(x,y)$ and $(x,z)$ in $\cT$
and $\cT'$, we can instead consider the cherries $(r,s)$ and $(r,a)$ in 
$\cT'$ and $\cT$, respectively, and, reversing the roles of $\cT$ and $\cT'$, apply
our arguments to these cherries instead. In particular, it follows
that we may assume $z \in L(\mathcal B)$.}

\orange{The last statements in the claim now follow immediately by Claims~3 and 9. $\blacksquare$}\\

\orange{Now, since $\cT$ contains the cherry $(x,y)$ and $|L(\cB)\cap L(\cT)|\geq 3$, we  \teal{will} assume that either 
$((x,y),p)$ or $((x,y),(p,q))$ is a pendant subtree of 
$\cT$ with $x,y,p,q \in L(\cB)$ distinct. Note that by Claim~5, the length of the path $\cN[\delta]$ is at least three. }

\orange{We next show that if the length of $\cN[\delta]$ is three, then (CP) holds, which completes the proof
of the theorem. Indeed, 
suppose that $\cN[\delta]$ has length three. Then without loss of generality $z=p$.}
\blue{Moreover, if $\cT$ contains
$((x,y),p)$, then setting $\mathcal S=((x,y),p)$, we
can apply a (C2)(a)(ii)-reduction to $\mathcal S$, remove
the edge $e''$, and apply  Lemma~\ref{reduction}(b) plus induction 
to see that (CP) must hold. And
if $\cT$ contains $((x,y),(p,q))$ then we must have $\cT = ((x,y),(p,q))$, 
in which case we can apply a (C2)(a)(iii)-reduction to $q$, remove
the edge $e''$,  and apply  Lemma~\ref{reduction}(b) plus induction 
to see that (CP) must hold.}
\begin{figure}[t]
	\center
		\scalebox{1.4}{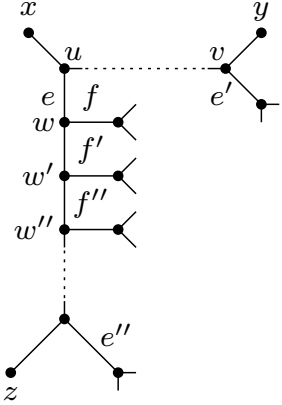}
	\caption{\blue{The configuration in $\cN$ that underpins the proofs of Claims~12--14 of Theorem~\ref{uppernet}.}
	\label{fig:theorem6-1-fig2}
 }
\end{figure}
\blue{Now, suppose that the length of the path $\cN[\delta]$ is greater than three.
Let $w$, $w'$ and $w''$ be the vertices in $\cN[\delta]$ that 
are distance two, three and four from $x$, respectively.
Let $f$, $f'$ and $f''$ be the edges in $\cN$ that are
not in $\cN[\delta]$ and that contain
the vertices $w$, $w'$ and $w''$, respectively (see Fig.~\ref{fig:theorem6-1-fig2})}.\\

\noindent \blue{{\bf Claim 12:} $f$ is in $\cN[\cT]$.}\\

\noindent \blue{{\bf Proof of Claim 12:} Suppose not. Then, since $f$ is not in $\cN[\cF']$ because $(x,z)$ is a cherry in $\cF'$, 
by Claim~4 $f$ must be a cut-edge of $\cN$. Moreover, $f$ 
is not a pendant edge otherwise $\cF'$ would contain 
a component that is a vertex which contradicts Claim~3.
Thus, $f'$ is not a cut-edge of $\cN$
as $\cB$ is a pendant blob,
and so $f'$ is contained in $\cN[\cT]$ by Claim~4.}

\blue{Now, if $z=w''$, then $\cT$ must contain $((x,y),z)$ 
as a proper pendant subtree and we can apply a (C2)(a)(ii)-reduction to $((x,y),z)$, remove
the edge $f'$, and apply  Lemma~\ref{reduction}(b) plus induction 
to see that (CP) must hold. And, if $z \neq w''$,
then as $f''$ is not a cut-edge of $\cN$ (since $f$ is) 
it follows that $f''$ is in $\cN[\cT]$ by Claim~4.
But then if 
$\cT$ contains $((x,y),p))$ as a pendant subtree then it
must be proper because two edges on $\cN[\delta]$ must be incident with $w''$. Hence,
we can apply a (C2)(a)(ii)-reduction to $p$, remove
the edge $f'$, and apply  Lemma~\ref{reduction}(b) plus induction 
to see that (CP) must hold. And, if $\cT$ contains
$((x,y),(p,q))$ as a pendant subtree, then either, without
loss of generality, $z=p$, or $f'$ is in the image of 
the path from $x$ to $p$ (or $q$). Furthermore, 
$((x,y),(p,q))$ must be proper because two of the edges that share a vertex with $f'$ or $f''$ are also contained in $\cN[\delta]$. Hence,
we can either apply a (C2)(a)(ii)-reduction to $q$ and remove $f''$
or a (C2)(a)(iii)-reduction to $(p,q)$ and remove $f'$ to
see that (CP) holds using Lemma~\ref{reduction}(b) and induction.
$\blacksquare$}\\

\noindent \blue{{\bf Claim~13:} If $\cT$ contains
$((x,y),z)$ as a pendant subtree, then $f$ is in the image of the pendant 
edge in $\cT$ that contains $p$ and $((x,y),z)$ is a proper
pendant subtree of $\cT$. }\\

\noindent \orange{{\bf Proof of Claim~13:}
Suppose that the claim is not true. Then since $e$ is an edge in $\cN[\cT]$ 
and $f$ is an edge in $\cN[\cT]$ by Claim~12, $\{w,w'\}$ is in the image of the pendant edge in $\cT$ that contains $p$, and $((x,y),p)$ is a proper pendant subtree of $\cT$. Now since every edge in $\cN[\delta]$
is contained in $\cN[\cT]$ by Claim~11, the edge $e''$ must be in the image of the pendant edge in $\cT$ that contains $p$. Moreover, $e''$ is not a cut-edge of $\cN$ since $p\in L(\cB)$. so we can apply a (C2)(a)(iii)-reduction to $p$, remove $e''$ and apply Lemma~\ref{reduction}(b) plus induction to see that (CP) holds. $\blacksquare$}\\


\noindent \blue{{\bf Claim~14:} If $\cT$ contains
$((x,y),(p,q))$ as a pendant subtree, then $f$ is in the image of the pendant 
edge in $\cT$ that shares a vertex with the cherry $(p,q)$ and $((x,y),(p,q))$ is a proper pendant subtree of $\cT$. 
 }\\
 
\noindent \blue{{\bf Proof of Claim 14:}
Suppose that the claim is not true. Then
using similar arguments to those used in the
proof of Claim~13, it follows that 
without loss of generality, $z=p$, and 
either $w''=z$ and $f'$ is in the image
of the pendant edge in $\cF'$ that contains $q$,
or $f'$ is a cut-edge that is not a pendant 
edge, and $f''$ is in the image
of the pendant edge in $\cF'$ that contains $q$.
In either case, $((x,y),(p,q))$ is a proper pendant subtree of $\cT$. Furthermore, in either case  by removing $f'$ or $f''$ respectively
and applying a (C2)(a)(iii)-reduction to $q$
we can use Lemma~\ref{reduction}(b) and induction
to see that (CP) holds.
$\blacksquare$}\\

\blue{Finally, to complete the proof of the theorem, 
if $\cT$ contains $((x,y),z)$ as a pendant subtree, then, by  Claim~13, we can apply a 
(C2)(a)(ii)-reduction to $p$, remove $f$ and
apply Lemma~\ref{reduction}(b) plus induction 
to see that (CP) must hold. And, if $\cT$ contains
$((x,y),(p,q))$ as a pendant subtree, then by Claim~14, we can apply a 
(C2)(a)(iii)-reduction to $(p,q)$, remove $f$ and
apply Lemma~\ref{reduction}(b) plus induction 
to see that (CP) must hold. 
\qed}

\section{Discussion}\label{sect:conclude}

In this paper, we have proven that the hybrid number of two binary phylogenetic trees
can be given in terms of cherry picking sequences.  There are several interesting 
future directions of research. For example, \teal{Humphries et al. \cite{HLS13} have considered a certain type of cherry picking sequence to characterize arbitrary size collections of rooted phylogenetic trees that can be displayed by a rooted phylogenetic network that is time consistent and tree-child.} \teal{It would be interesting to investigate extensions of our results to computing the hybrid number of an arbitrary collection
of not necessarily binary phylogenetic trees (or forests) 
that can be displayed by an unrooted phylogenetic network of a given class such as tree-child or orchard (see \cite{DL24} for the 
definition of such networks)}. 

As mentioned in the introduction, the hybrid number of two 
unrooted binary phylogenetic trees is equal to the TBR distance between the two trees. Recently, there
has been some interest in improving algorithms to compute the TBR distance
(see e.g. \cite{LK19,K22}), and some of the data reductions rules introduced  in these
papers are similar to  the rules that we used to define cherry picking sequences 	
(for example, the reduction described 
in (C2)(a)(ii) is similar to the (*,3,*)-reduction in \cite[Section 3.1]{LK19}). 
It would be interesting
to see if there is a deeper connection between these two concepts, and
to investigate whether, for example, the results presented here could be used
to improve depth-bounded search tree algorithms to compute the TBR distance.
To investigate this connection,  a first step might be to understand whether or not
all of the reductions in the definition of a cherry picking sequence are needed, or whether a somewhat simpler list of reductions might be found to define a cherry picking sequence so that Theorem~\ref{main} still holds.

Finally, we have considered the hybrid number $h(\cF,\cF')$ for $\cF$ and $\cF'$ \blue{being} an 
arbitrary pair of \blue{binary} forests on $X$. In case we restrict $h$ to pairs of phylogenetic trees on $X$, 
$h$ is the TBR distance on the set of phylogenetic trees on $X$. It  would be 
interesting to understand properties of $h$ 
for pairs of forests in general (e.\,g.\,is $h$ related to some kind of TBR distance 
on the set of forests on $X$?).

\acknowledgements
KTH and VM would like to thank the School of Computer Science at the University of Auckland for hosting them in December 2019. Part of this paper is based upon work supported by the National Science Foundation under Grant No. DMS-1929284 while all authors were in residence at the Institute for Computational and Experimental Research in Mathematics in Providence, Rhode Island, US, during the {\it Theory, Methods, and Applications of Quantitative Phylogenomics} program. All authors would like to thank the anonymous referees for numerous comments and suggestions to improve the exposition of the paper.


\end{document}